\documentclass[11pt]{article}
\usepackage[pagewise]{lineno}
\usepackage{amsfonts}
\usepackage{amsfonts}  
\usepackage{amsthm}
\usepackage{amsmath}
\usepackage{color}

\title{Extremal functions for an embedding from  some anisotropic  space, and partial differential equation involving  the "one  Laplacian" }
\author{ Fran\c coise Demengel, Thomas Dumas 
\footnote{Laboratoire AGM, UMR 80-88, Universi\'e de Cergy Pontoise, 95302 Cergy Pontoise. }}

\date{}

\catcode`@=11 \@addtoreset{equation}{section} \catcode`@=12

\newtheorem{theo}{Theorem}[section]
\newtheorem{prop}[theo]{Proposition}
\newtheorem{rema}[theo]{Remark}
\newtheorem{defi}[theo]{Definition}
\newtheorem{cor}[theo]{Corollary}
\newtheorem{lemme}[theo]{Lemma}

\def\R{\mathbb  R}

\newcommand{\N}{{\bf N}}

\begin{document}
\maketitle

\begin{abstract}

 In this paper, we prove the existence of extremal functions for the best constant of  embedding from anisotropic  space,  allowing  some of the Sobolev exponents to be equal to $1$. 
  We prove also that the extremal functions satisfy a partial differential equation  involving the $1$ Laplacian. 
   \end{abstract}
    \section{Introduction}
    Anisotropic Sobolev spaces have been studied for a long time, with different purposes. Let us recall that  for $\vec p = ( p_1, \cdots ,p_N)$, and the $p_i \geq 1$ the space 
    ${\cal D}^{1, \vec p}(\R^N)$, denotes  the closure of ${ \cal D} ( \R^N)$ for the norm 
    $ \sum_i |\partial_i u |_{p_i}$.     The existence of a critical  embedding from ${\cal D}^{1, \vec p}( \R^N)$ into $L^{ p^\star}$, with $p^\star = { N \over \sum_i { 1 \over p_i}-1}$ when $  \sum_i { 1 \over p_i}>1$  is due to Troisi, \cite{T}.

        There is by now a large number of papers and an increasing interest about anisotropic problems. With no hope of being complete, let us mention some pioneering works on anisotropic Sobolev spaces \cite{KruKoro}, \cite{Ni} and some more recent regularity results for minimizers of anisotropic functionals, that we will cite below.

      Let us note that anisotropic operators  bring new problems, essentially when one  wants to prove regularity properties.        As an example the property that $\Omega$ be Lipschitz does not ensure the embedding    $W_o^{1, \vec p}( \Omega )\hookrightarrow L^{ p^\star}( \Omega)$. This is linked to the fact that in  the absence of further geometric properties of $\Omega$,   one cannot provide a continuous extension operator from $W^{1 , \vec p} (  \Omega)$ in ${ \cal D} ^{1, \vec p} ( \R^N)$.  To illustrate this,  see the counterexample in \cite{KrKo}, see also 
       \cite{ThD}  for  one example when some of the  $p_i$ are equal to $1$, in the context of the present article.

Let us say a few words about the existence and regularity results of solutions to  $-\sum_i \partial_i ( | \partial_i u |^{ p_i-2} \partial_i u ) = f$,     $u=0$ on $\partial \Omega$  when $\Omega$ is  a bounded domain in $\R^N$.     Assuming  a convenient assumption on $f$,  the existence of solutions can generally  easily be obtained  by the use of classical methods in the calculus of variations.    But,   as a first step in the  regularity  of such solutions, the local boundedness of the solutions,  can fail   if the supremum of the $p_i$ is too large,
   let us cite to that purpose  \cite{G} and \cite{M1} where the author  exhibits a counterexample to the local boundedness when  $p_i=2$ for $i\leq N-1$ and 
     $p_N  > 2{N-1\over N-3}$. 
      This   restriction on $\vec p$, to ensure the local boundedness  is confirmed by the results obtained later  : let us cite in a non exhaustive way      \cite{CMM}, \cite{M1}, \cite{BMS}.   
       From all these papers it emanates  in a first time that a  sufficient condition for a local minimizer  to be locally bounded is that 
   the supremum of the $p_i$ be  strictly   less than  the critical exponent  $ p^\star$. 
       This  local  boundedness   is extended  by Fusco Sbordone  in  \cite{FS}  to the case where  $ \sup p_i =  p^\star$. 
   For further regularity properties of the solutions, as  the local higher  integrability of   the local minimizers for  some  genarized functionals, 
       see   Marcellini in  \cite{M2},   and  Esposito Leonetti Mingione    \cite{ELM1,  ELM2} . 
       
        \bigskip

         Coming back to ${ \cal D}^{1 , \vec p} ( \R^N)$,     and        concerning extremal functions, let us recall that in the isotropic case, the first results concerned  the case where $p_i  = 2$ for all $i$, in which  case the extremal functions  are   solutions  of 
            $-\Delta u = u^{2^\star-1}$.
             The existence and the explicit form  of them  is   completely solved by Aubin \cite{A}, and Talenti, \cite{Ta}. 
             For $W^{1,p}$ and the isotropic $p$ Laplacian, say 
              $ -\Delta_p u = -{ \rm div} ( | \nabla u |^{ p-2} \nabla u)$ the explicit form is also known as  the family of radial  functions 
              $u_{a, b} ( r) = ( a+ b r^{ p\over p-1})^{p-N \over p}, $
                           while  for the $p$-Laplacian non isotropic,  say  for the equation 
                $-\sum_{i=1}^N \partial_i ( | \partial_iu|^{ p-2} \partial_i u  ) = u^{ p^\star-1}$,  the explicit solutions are obtained by Alvino Ferrone Trombetti  Lions \cite{AFTL} and are given by 
                $u_{a, b} ( r) = ( a+ b \sum_{i=1}^N |x_i|^{ p\over p-1})^{p-N \over p}.$
                 For  further results about sharp  embedding constant, and a new, elegant approach by using  mass transportation the author can see \cite{CNV}.  
                
                \bigskip
                Let us now consider the case where the  $p_i$  can be different from each others,  and  let us  first cite  the paper of Fragala Gazzola and Kawohl \cite{FGK}, where the authors prove the existence of extremal functions for some  subcritical embeddings in the case of bounded domains. 
        
      For the case of $\R^N$ and the critical case, 
             the existence of extremal functions  is proved in \cite{HR2},    when all the $p_i >1$, and $p^+:= \sup p_i < p^\star$.  The authors  provide also some properties of the extremal functions, as the $L^\infty$ behaviour, extending in that way the regularity  results already obtained for solutions of  anisotropic partial differential equation  in  a bounded domain, with a right hand side sub-critical as in \cite{FGK}, to the critical one.  The method uses essentially the concentration compactness theory of P. L. Lions \cite{PL1, PL2} adapted to this context, and some other tools developed also in  a more general context in \cite{HR1} .
             
              In the case where $p^+ = p^\star$ and for more general domains than $\R^N$ the reader can see Vetois, \cite{V}. In this article  this author provides also some  vanishing properties of the solutions, as well as some further regularity properties of the solutions. 
              
               When some of the $p_i$ are equal to $1$, let us cite the paper of Mercaldo, Rossi,  Segura de leon, Trombetti, \cite{MRST}, which  proved the existence of solutions in some anisotropic  space, with some derivative in the space of bounded measures,  for the $\vec p$-Laplace equation in  bounded domains,  using the definition of the one Laplacian with respect to the coordinates for which $p_i=1$. For the existence of extremal functions in the case of $\R^N$, and in the best of our knowledge, nothing has been done in the case where some of the $p_i$ are equal to $1$. Of course  in that case these extremal functions have their corresponding derivative in the space $M^1 ( \R^N)$ of  bounded measures on $\R^N$. 
 Even if the existence of such extremal can be obtained following the lines in the proof of \cite{HR2}, the partial differential equation  satisfied by the extremal cannot be obtained   by this existence's result. 
               In order to get it,         we are led to consider  a  sequence of  extremal functions for the embedding of ${ \cal D}^{1, \vec p_\epsilon} (\R^N)$ in $L^{p_\epsilon^\star}(\R^N)$ where in $\vec p_\epsilon$, all the $p_i^\epsilon > p_i$ and tend to them as  $\epsilon$ goes to zero.   Note that one of the difficulties  raised by this approximation is that, due to the unboundedness of $\R^N$,  ${ \cal D}^{1, \vec p_\epsilon} ( \R^N)$ is not a  subspace of ${ \cal D}^{1, \vec p} ( \R^N)$, a problem which does not appear when one  works with bounded domains, see \cite{ThD}.  In particular this does not allow to use directly the  concentration compactness theory of P.L.  Lions, \cite{PL1}.  
                 We will prove both that the best constant for the  embedding  from ${ \cal D}^{1, \vec p_\epsilon} ( \R^N)$ in $L^{p_\epsilon^\star}( \R^N)$ converges to the best constant   for the embedding of ${ \cal D}^{1 , \vec p} ( \R^N)$ into $L^{p^\star}(\R^N)$,  and that  some extremal  $u_\epsilon$ converge sufficiently tightly to some $u$. Passing to the limit in the partial differential equation satisfied by $u_\epsilon$ one obtains  that $u$ is  extremal and satisfies the required partial differential equation.

 \section{Notations, and previous results}
 
\subsection{Some measure Theory,  definition and properties of the space $BV^{ \vec p}$}

     \begin{defi}
             Let $\Omega$ be an open set in $\R^N$, 
             and  $M (\Omega)$,  the space of  scalar Radon measures, i.e.  the  dual of ${ \cal C}_c ( \Omega)$. 
             Let 
              $M^1( \R^N)$ be  the space of scalar  bounded Radon measures  or equivalently the subspace of $\mu \in M( \Omega)$ which satisfy 
               $\int_\Omega | \mu| = \sup _{ \varphi \in { \cal C}_c ( \Omega)}  \langle \mu, \varphi\rangle< \infty$.
               
                 $M^+ (\Omega)$ is the space of non negative bounded measures on $\R^N$. 
                 \end{defi}
                 
                 \begin{defi}
                  When $\mu = (\mu_1, \cdots ,\mu_n)$ we define 
                  $ | \mu | =( \sum \mu_i^2)^{1 \over 2}$ as the measure  :
                  For  $\varphi \geq 0$ in ${ \cal C}_c( \Omega)$,
                  $\langle | \mu|, \varphi\rangle = \sup_{ \psi \in { \cal C}_c( \Omega, \R^N), \sum_1^N \psi_i^2 \leq \varphi^2} \sum \langle \mu_i, \psi_i\rangle $.
                  \end{defi}

               Let us recall that 
               \begin{defi}
               
                $\mu_n\rightharpoonup \mu$ vaguely or weakly in $M( \Omega )$ if 
                for any $\varphi \in { \cal C}_c ( \Omega)$, 
                 $\langle \mu_n, \varphi\rangle \rightarrow \langle \mu, \varphi \rangle$.

                 When  $\mu_n$ and $\mu$ are  in $M^1 ( \Omega)$ we will say that  $\mu_n$ converges tightly  to $\mu$ if    for any $\varphi \in { \cal C} ( \R^N)$ and bounded, 
                     $\langle \mu_n, \varphi\rangle \rightarrow \langle \mu, \varphi \rangle$. 
                     \end{defi}

                    \begin{rema}\label{remvague}
                    When $\mu_n \geq 0$, the tight convergence of $\mu_n$ to $\mu$ is equivalent to  both the two conditions
                        1)  $\mu_n\rightharpoonup \mu$ vaguely  and 2) $\int_{ \Omega} \mu_n \rightarrow \int_{ \Omega} \mu$. 
                         \end{rema}
                        
                         We will frequently use the following density result: 
                         
                          \begin{prop}
                           
                            If $\vec \mu\in M^1 ( \Omega, \R^N)$  there exists $u_n \in { \cal D} ( \Omega,\R^N)$  such that   $(u_n)_i^\pm $, respectively $| u_n |$  converges tightly to $\mu_i^\pm$ (respect. $| \mu |$).                               \end{prop} 
                         
                                            The reader is referred to \cite{Di}, \cite{DT},  for further properties  on convergence of measures and  density of  regular functions for the vague and  tight topology. 
                        \bigskip

Let $N_1\leq N \in \N$, and $\vec{p}:= (p_1,\cdots, p_N) \in \R^N$ such that $p_i=1$ for all $1\leq i \leq N_1$, and $p_i >1$ for all $N_1 +1 \leq i \leq N$.

Let  $p^+ = \sup p_i$,  and
   $$p^* := {N\over N_1 + \sum_{i=N_1 + 1}^N {1\over p_i} - 1}.$$ In all the paper we will suppose that 
   $p^+ < p^\star$. 
 Let   $\mathcal{D}^{1,\vec{p}}(\R^N)$ be  the completion of $\mathcal{D}(\R^N)$ with respect to the norm   
  \begin{equation} \label{normp} |u|_{ \vec p} =|  (\sum_{i=1}^{N_1} (  \partial_i u)^2)^{1\over 2} |_1 + \sum_{i=N_1+1} ^N | \partial_i u |_{p_i}:= | \nabla_1 u|_1 +  \sum_{i=N_1+1} ^N | \partial_i u |_{p_i}
  \end{equation} 
  where  $\nabla_1 u$  is the $N_1$ vector $( \partial_1 u, \cdots, \partial_{N_1} u)$, and   $|u|_{p_i}$ denotes  for $i\geq N_1+1$ the usual $L^{p_i}( \R^N)$ norm.
     \begin{rema}
      Of course by the equivalence of norms in $\R^{N_1}$ this completion coincides with the completion for the norm 
      $\sum_{i=1}^N | \partial_i u |_{p_i}.$
      \end{rema}

  We now recall the existence of the  embedding from ${ \cal D}^{1, \vec p}( \R^N)$ in $L^{ p^\star}( \R^N)$,  
 a particular case of the  result of Troisi , \cite{T}. 
 \begin{theo}\label{troisi} 
 $$\mathcal{D}^{1,\vec{p}}(\R^N) \hookrightarrow L^{p^*}(\R^N),$$

and there exists some  constant $T_0$  depending only on $\vec{p}$, and $N$ such that

\begin{equation}\label{eqemb}T_0 |u|_{p^*} \leq \prod_{i=1}^N |\partial_i u|_{p_{i}}^{1\over N}, \text{ and } |u|_{p^*} \leq {1\over T_o N} \left( \sqrt{N_1} |\nabla_1 u|_1 + \sum_{i=N_1+1}^N |\partial_i u|_{p_i}\right), \end{equation} 

for all $u\in \mathcal{D}^{1,\vec{p}}(\R^N)$.

\end{theo}      

We now introduce a weak closure of ${ \cal D} ( \R^N)$ for the norm (\ref{normp}).
            Set
                   
              \begin{eqnarray*}\label{defBV}
              BV^{\vec{p}}(\R^N):&=& \{ u\in L^{p^*}(\R^N), \partial_i u \in M^1(\R^N) \text{ for }1\leq i \leq N_1, \text{ and } \partial_i u \in L^{p_i}(\R^N)\\
              && \text{ for } N_1 +1\leq i \leq N \}. 
              \end{eqnarray*}   
               We also define 
               $$BV^{ \vec p}_{loc} (\R^N)= \{ u\in { \cal D}^\prime ( \R^N),\  \varphi u \in BV^{ \vec p} ( \R^N), \ {\rm for\ any\ } \varphi \in { \cal D} ( \R^N)\}$$
            
               \begin{defi}
                We will say that $u_n \in BV^{ \vec p} (  \R^N)$ converges weakly to $u$ if 
                $ u_n \rightharpoonup  u $ (weakly)  in $L^{ p^\star}$, 
                                $\partial_i u_n $ converges vaguely to $\partial_i u $ in $M^1 ( \R^N)$ when $i \leq N_1$, and 
                 $\partial_i u_n \rightharpoonup \partial_i u $  (weakly)  in $L^{ p_i}$, when $i > N_1$. 
                 . 
                  
                  The convergence is said to be tight if furthermore 
                  $ \int_{ \R^N} | \partial_i u_n |^{p_i} \rightarrow \int_{ \R^N}  | \partial_i u | ^{p_i}$ for any $i  \geq N_1$ , 
               and 
                   $\int_{ \R^N} |  \nabla_1 u_n | \rightarrow \int _{ \R^N} | \nabla_1 u |$. 
                  \end{defi}
                  \begin{rema}
                  If $u_n$ converges weakly to $u$,  since $(u_n)$ is bounded in $L^{ p^\star}$,  it converges strongly in $L^q_{loc} $ for a subsequence, when  $q < p^\star$ and then for a subsequence it converges almost everywhere. 
                   \end{rema}

              \begin{prop}\label{propBV}
       It is equivalent to say that 
       \begin{enumerate}
       \item $u \in BV^{ \vec p} ( \R^N)$ 
      \item There exists $u_n \in { \cal D} ( \R^N)$ which converges tightly to $u$. 
                     \item 
               There exists $u_n \in { \cal D} ( \R^N)$  which converges weakly to $u$.
                            \end{enumerate}
                             \end{prop} 
                            
                              \begin{rema}
                              
                               Following the lines in the proof below,  but using  strong convergence  in $L^1$ of  $\partial_i u_n$ for $i\leq N_1$,  in place of tight convergence, 
                               it is clear that ${\cal D}^{1,\vec p}(\R^N)= \{ u\in L^{ p^\star} ( \R^N), \partial_i u \in L^1(\R^N), \ i \leq N_1, \ \partial_i u \in L^{ p_i}(\R^N), \ i \geq N_1+1\}$.
                                \end{rema}

\begin{proof} 

 Suppose that 1) holds. 

We begin by a troncature. For $1\leq i \leq N$ let $\alpha_i$ defined as 

     $$ \alpha_i = { p^\star \over p_i}-1. $$
           Let  $\varphi \in { \cal D} ( ]-2,2[)$, $\varphi = 1$  on  $[-1,1]$, and  for all $n\in \N$, 
         $$u_n ( x) = \Pi_{i=1}^N \varphi ( {x_i \over  n^{ \alpha_i}}) u(x).$$
             We  denote $C_n =    \Pi_{i=1}^N [-2 n^{ \alpha_i}, 2 n^{ \alpha_i}]$, note that $| C_n | = 4^N n^{ \sum_{i=1}^N \alpha_i} $. 
         We need to prove that $\partial_i u_n \rightarrow \partial_i u $ in $L^{p_i}(\R^N)$ for all $i\in [N_1+1, N]$.
Since 
$$\partial_i u_n(x) = u(x) \partial_i \left( \Pi_{j=1}^N \varphi ( {x_j \over  n^{ \alpha_j}})\right) + \Pi_{j=1}^N \varphi ( {x_j \over  n^{ \alpha_j}}) \partial_i u(x),$$
it is sufficient to prove that
          $ u \partial_i \left( \Pi_{j=1}^N \varphi ( {x_j \over  n^{ \alpha_j}})\right) \rightarrow 0$ in $L^{p_i}(\R^N)$.       By  H\"older's inequality
           \begin{eqnarray*}
            \int_{\R^N}  |u \partial_i \left( \Pi_{j=1}^N \varphi ( {x_j \over  n^{ \alpha_j}})\right) |^{p_i} &\leq 
            & {c \over n^{ \alpha_i p_i} } (\int_{ \R^{N-1} }\int _{ n^{ \alpha_i} \leq |x_i| \leq 2 n^{ \alpha_i}} |u|^{p^\star}  )^{ p_i \over p^\star} | C_n | ^{1-{p_i \over p^\star}}\\
             & \leq & c' n^{- \alpha_ip_i +  ( 1-{p_i \over p^\star})\sum_{j=1}^N \alpha_j } o(1)
             \end{eqnarray*}
             which tends to zero, 
            since  $u \in L^{ p^\star}(\R^N)$ implies that  $\int_{ \R^{N-1}}\int _{ n^{ \alpha_i} \leq |x_i| \leq 2 n^{ \alpha_i}} |u|^{p^\star} \underset{n\to \infty}{\rightarrow} 0$,  
             and 
 for any $i$ by the definition of $\alpha_i$,                        $  \alpha_i  p_i \geq ( 1-{p_i \over p^\star}) \sum_{j=1}^N \alpha_j $ . 
              In the same manner we have  
 $ \int_{ \R^N}  | \nabla_1 u _n-\nabla_1 u|\rightarrow 0$. 
 The second step  classically  uses  a regularisation process. Recall that 
 that when $\vec \mu$ is a  compactly supported  measure  in $\R^N$, with values in $\R^N$,     when  $\rho\in { \cal D} ( \R^N)$, $\int \rho=1$, $\rho\geq 0$,  and  $\rho_\epsilon= {1 \over \epsilon^N} \rho({x\over \epsilon})$,  
   $\rho_\epsilon \star |\vec \mu|\ {\rm converges \ tightly \ to } \ |\vec \mu|, \ \rho_\epsilon \star \mu_i^{\pm}  \ {\rm converges \ tightly \ to } \ \mu_i^{\pm}. $
  From this one derives  the tight convergence when $\epsilon$ goes to zero and $n$ to $\infty $ of  $| \nabla_1 ( \rho_\epsilon \star u_n) |$ towards $|\nabla_1 u|$. 
 
   2) implies 3)    is obvious. To prove that 3) implies 1), note that if $(u_n)$ is weakly convergent to $u$, one has the existence of some constant independent on $n$ so that 
   $ | \nabla_1 u_n |_1 +\sum_{i=N_1+1}^N| \partial_i u_n |_{p_i} \leq C$. Then by the embedding in Theorem \ref{troisi},   $(u_n)$ is bounded in $L^{p^\star}$, and by extracting subsequences from $\nabla_1 u_n$ in $M^1( \R^N, \R^{N_1} )$ weakly and from   $\partial_i u_n$   in $L^{p_i}$  weakly for $i \geq N_1+1$,  one gets that  the limit $u \in BV^{ \vec p} ( \R^N)$.
   \end{proof}

\begin{rema}\label{embBV}
Using the last proposition, one sees that  \eqref{eqemb}extends to the functions in $BV^{ \vec p} ( \R^N)$.  
\end{rema} 

We  now enounce a result which extends the definition of the "Anzelotti pairs", \cite{Anz}, see also Temam \cite{Tem}, Strang Temam in  \cite{ST}, and  \cite{FD1, FD2, FD3}.
 \begin{theo}\label{sigma1grad1}
 Let $\sigma$  a function with values in $\R^N$,  such that 
 its projection   ${\sigma^1}$ on the first $N_1$ coordinates, belongs to $L^\infty_{loc} ( \R^N, \R^{N_1})$, and suppose that for any $i \geq N_1+1$, $\sigma \cdot e_i \in L_{loc} ^{ p_i^\prime}$, and that 
 ${ \rm div} \sigma \in L^{ p^\star \over p^\star-1}_{loc}$. Then if  $u \in BV^{\vec p}_{loc} ( \R^N)$, one can define a distribution 
 $\sigma \cdot \nabla u$ in the following manner,  
  for  $\varphi \in { \cal D} ( \R^N)$, 
  $$ \langle \sigma \cdot \nabla u, \varphi \rangle = -\int_{ \R^N}{ \rm div}  \sigma   ( u\varphi) -\int_{ \R^N} (\sigma \cdot \nabla \varphi) u.$$
   Then $ \sigma \cdot \nabla u$ is  a  measure,  and  
    ${\sigma^1} \cdot \nabla_1  u:=  \sigma \cdot \nabla u - \sum _{ i = N_1+1}^N\sigma_i \partial_i u $   is  a measure absolutely continuous with respect to $|\nabla_1 u|$, with  for $\varphi \geq 0$ in ${ \cal C}_c ( \R^N)$ : 
     \begin{equation}\label{ac} \langle | {\sigma^1}\cdot \nabla_1 u|,\varphi\rangle  \leq |{\sigma^1}| _{L^\infty( Suppt \varphi)}\langle   |\nabla_1 u|,  \varphi\rangle .
     \end{equation}        Furthermore   when ${\sigma^1} \in L^\infty( \R^N, \R^{N_1})$ and $\sigma_i \in L^{p_i^\prime }( \R^N)$,  for any $i \geq N_1+1$, ${ \rm div} \sigma  \in L^{ p^\star \over p^\star-1}( \R^N)$ and $u \in BV^{ \vec p}( \R^N)$, $\sigma \cdot \nabla u$ and ${\sigma^1} \cdot \nabla_1u$ are  bounded measures on $\R^N$ and  one has  
      \begin{equation}\label{intout}
       \int _{ \R^N}\sigma \cdot \nabla u = -\int_{ \R^N}  { \rm div} ( \sigma) u
       \end{equation}
        and 
        $$| \sigma^1\cdot \nabla_1 u | \leq | \sigma^1|_\infty | \nabla_1 u |$$
      \end{theo}
      
       \begin{proof} 
       Take $\psi\in { \cal D} ( \R^N)$, $\psi = 1$ on Suppt $\varphi$. Then if $u\in BV_{loc}^{\vec p} ( \R^N)$, $\psi u \in BV^{\vec p} (\R^N)$. By Proposition \ref{propBV}, there exists $u_n\in { \cal D} ( \R^N)$ such that $u_n$ converges tightly to $\psi u$  in $BV^{ \vec p} ( \R^N)$. 
  By  the classical Green's formula 
        $$ \int_{ \R^N} \sigma \cdot \nabla u_n \varphi = - \int_{ \R^N}{ \rm div} ( \sigma )  ( u_n\varphi) -\int _{ \R^N}(\sigma \cdot \nabla \varphi) u_n.$$
        Using the weak convergence of $u_n$ towards $\psi u$ one gets that $\int (\sigma\cdot\nabla u_n)\varphi$ converges to 
        $\langle\sigma\cdot \nabla u, \varphi \rangle$. By the assumptions on $\sigma_i$ and $\partial_i u_n$, one  has $\int \sigma_i \partial_i u_n \varphi \rightarrow \int \sigma_i \partial_i u \varphi$,  for $i \geq N_1+1$, hence $\int (\sigma^1 \cdot \nabla_1 u_n) \varphi\rightarrow \langle\sigma^1\cdot \nabla_1 u, \varphi \rangle$. Furthermore, using  for $\varphi \geq 0$,  
        $|\int \sigma^1 \cdot \nabla_1 u_n \varphi | \leq |\sigma^1|_{ L^\infty ( Suppt \varphi)} \int |\nabla_1 u_n |  \varphi\rightarrow  |\sigma^1|_{ L^\infty ( Suppt \varphi)} \int |\nabla_1 u |  \varphi$,  one gets (\ref{ac}). 
         The  identity ( \ref{intout})  is easily obtained by letting $\varphi$ go to $1_{\R^N}$, since all the measures involved are bounded measures.
         \end{proof}

     \bigskip
   \subsection{The approximated space ${ \cal D}^{1, \vec p_\epsilon}( \R^N)$}  
     
 Let $\epsilon >0$ small,  define 
 \begin{equation}\label{defai} a_i^\epsilon = { (p_i-1) p_i \epsilon^2\over 1-\epsilon (p_i-1)},\  
 \epsilon_i = p_i \epsilon + a_i^\epsilon\ 
 {\rm  and }\ p_i^\epsilon = p_i ( 1+ \epsilon_i).
 \end{equation} 
    Note that one has 
     for all $i \geq N_1+1$, 
     $ { p_i ( 1+ \epsilon_i) \over \epsilon_i} = { 1+ \epsilon \over \epsilon}$. 
      We define ${ \cal D}^{1, \vec p_\epsilon}(\R^N)$ as the closure of ${ \cal D}(\R^N)$ for the  norm 
        $ | \nabla_1 v |_{1+ \epsilon} + \sum_{i=N_1+1}^N| \partial_i v |_{p_i^\epsilon}$.
    Then 
      the critical exponent  $p_\epsilon^\star$ for this space is  defined by       ${ N \over p_\epsilon^\star}  = { N_1 \over 1+ \epsilon} + \sum_i { 1 \over p_i^\epsilon}-1$. 
            Note that 
       $p_\epsilon^\star$ satisfies 
       $${ N \over p_\epsilon^\star} = { N \over p^\star}-{ \epsilon N \over 1+ \epsilon}, $$
        and as soon as $\epsilon$ is small enough, $p^+_\epsilon < p_\epsilon^\star$. 
       Let us finally define 
      \begin{equation}\label{eqlambda}
      \lambda_\epsilon = { p_\epsilon^\star \epsilon \over 1+ \epsilon} + 1,
      \end{equation} 
     and note for further purposes that $\lambda_\epsilon p^\star = p_\epsilon^\star$. 
     
    Recall that   as a consequence of the embedding of Troisi, \cite{T} one has

   $$\mathcal{D}^{1,\vec{p_\epsilon}}(\R^N) \hookrightarrow L^{p_\epsilon^*}(\R^N),
    $$
      and there exists some $ T_0^\epsilon>0$, such that for all $u \in { \cal D}^{1, \vec p_\epsilon}( \R^N)$, 
   $$T_0^\epsilon |u|_{p_\epsilon^*} \leq \prod_{i=1}^N |\partial_i u|_{p_{i}^\epsilon}^{1\over N}, \text{ and  then } |u|_{p_\epsilon^*} \leq {1\over N T_0^\epsilon } ( \sqrt{N_1}  | \nabla_1 u | _{1+\epsilon} + \sum_{i=N_1+1}^N |\partial_i u|_{p_i^\epsilon} )
$$
for all $u\in \mathcal{D}^{1,\vec{p_\epsilon}}(\R^N)$.

     \bigskip

  Let us define 
          $${\cal K}_\epsilon= \inf_{u\in { \cal D}^{1, \vec p_\epsilon}(\R^N), |u|_{ p_\epsilon^\star} = 1} \left[{1 \over 1+ \epsilon} | \nabla_1 u |_{1+ \epsilon}^{1+ \epsilon} + \sum_{i=N_1+1}^N { 1 \over p_i^\epsilon} | \partial_i u |_ {p_i^\epsilon} ^{p_i^\epsilon} \right]$$
and  
 $${\cal K} = \inf_{u\in { \cal D}^{1, \vec p}(\R^N), |u|_{ p^\star} = 1} \left[ | \nabla_1 u |_{1} + \sum_{i=N_1+1}^N { 1 \over p_i} | \partial_i u |_ {p_i} ^{p_i}\right] $$

It is clear  by Proposition \ref{propBV} that 
$$ {\cal K} =  \inf_{u\in BV^{\vec p}(\R^N), |u|_{ p^\star} = 1} \left[  \int | \nabla_1 u | + \sum_{i=N_1+1}^N { 1 \over p_i} | \partial_i u |_ {p_i} ^{p_i} \right]$$
Adapting the proof in  \cite{HR2} one has the following result
 \begin{theo}
  There exists $u_\epsilon\in { \cal D}^{1, \vec p_\epsilon}( \R^N)$  non negative which satisfies $|u_\epsilon |_{ p_\epsilon^\star} = 1$ and 
  $$ {\cal K} _\epsilon  =  {1 \over 1+ \epsilon} | \nabla_1 u_\epsilon |_{1+ \epsilon}^{1+ \epsilon} + \sum_{i=N_1+1}^N { 1 \over p_i^\epsilon} | \partial_i u_\epsilon |_ {p_i^\epsilon} ^{p_i^\epsilon} .$$
   Furthermore there exists $l_\epsilon>0,$ so that  
    \begin{equation}\label{extremalepsilon} -\sum_{i=1}^{N_1} \partial_i ( | \nabla_1 u_\epsilon|^{ \epsilon-1}\partial_i u_\epsilon) -\sum_{ i=N_1+1}^N \partial_i ( | \partial_i u_\epsilon |^{ p_i^\epsilon-2} \partial_i u_\epsilon ) = l_\epsilon u_\epsilon ^{ p_\epsilon^\star-1}.
    \end{equation} 
     
     \end{theo}
     
      \bigskip
      
      In the sequel we will use the notation ${\rm div}_1 $ as the divergence  of some $N_1$ vector  with respect to the $N_1$ first variables. 
    
      By multiplying equation ( \ref{extremalepsilon}) by $u_\epsilon$ and integrating one has 
      ${\cal K} _\epsilon \leq l_\epsilon \leq  p_\epsilon^+ {\cal K} _\epsilon$ , and as we will see in Proposition \ref{limsup} that $\limsup {\cal K}_\epsilon \leq {\cal K}, $
    if   $u_\epsilon$ is an extremal function for ${\cal K} _\epsilon$, 
  $ | \nabla_1 u_\epsilon |_{1+ \epsilon} $ and $| \partial_i u_\epsilon |_ {p_i^\epsilon} $ are bounded independently on $\epsilon$, hence one can extract from it a subsequence which converges weakly in $BV^{ \vec p}$.
    In the sequel we will prove that by choosing conveniently the sequence $u_\epsilon$,  it converges  up to subsequence to an extremal function for ${\cal K}$.

  \section{The main results}
  
   The main result of this paper is the following :
    
    \begin{theo}\label{exiext}
    1)  There exists $v_\epsilon\in { \cal D}^{1, \vec p_\epsilon}( \R^N) $, $|v_\epsilon|_{ p^\star_\epsilon} = 1$,  an extremal function for ${\cal K} _\epsilon$,  which converges  in the following sense  to $v\in BV^{ \vec p} ( \R^N)$: 
      $v_\epsilon$ converges to $v$   in the distribution sense,   and almost everywhere,      $\int_{ \R^N}  | \nabla_1 v_\epsilon |^{1+ \epsilon} \rightarrow \int_{ \R^N} | \nabla_1 v | $, 
       $\int_{ \R^N} | \partial_i v_\epsilon |^{ p_i^\epsilon} \rightarrow \int_{ \R^N} | \partial_i v |^{p_i}$  for  all $i \geq N_1+1$,  and 
       $ |v|_{ p^\star} = 1$. Furthermore $\lim {\cal K}_\epsilon = {\cal K} $.  As a consequence $v$ is an extremal function for ${\cal K} $.

       2)  $v$ satisfies the partial differential equation :   
       \begin{equation} \label{partial differential equation }-{ \rm div}_1 ( {\sigma^1}) -\sum_{ i= N_1+1}^N \partial_i ( | \partial_i  v |^{p_i-2} \partial_i v) = lv^{ p^\star-1}, \ {\sigma^1} \cdot \nabla_1 v = | \nabla_1 v | 
       \end{equation}
        where ${\cal K} \leq l\leq p^+ {\cal K} $. 
       \end{theo}

        The proof of Theorem \ref{exiext} is given in the next subsection, and it relies of course on a convenient adaptation of the PL Lions compactness concentration theory.  
         However, due to the fact that the exponents of the derivatives and the critical exponent vary with  $\epsilon$, we are led to introduce a power  $v_\epsilon ^{ \lambda_\epsilon}$ of some convenient extremal function,  - where  $\lambda_\epsilon $ has been defined in (\ref{eqlambda})-,  and to  analyze the behaviour of this new sequence, which belongs to $BV^{\vec p}( \R^N)$,  and is bounded in that space, independently on $\epsilon$, as we will see later.         
          \bigskip
          
           In a second time we prove that 
           
            \begin{theo}\label{Linfty}
           Let $v$ be given by Theorem \ref{exiext}. Then 
           $v\in L^\infty ( \R^N)$  and there exists some constant $C( |v|_{p^\star})$ depending on the $L^{p^\star}$ norm of $v$ and on universal constants, such that 
           $|v|_\infty \leq C ( |v|_{ p^\star})$. 
           \end{theo}
           
           \subsection{Proof of Theorem \ref{exiext}}
           
           The proof is the consequence of several lemmata and propositions. 
           \begin{lemme}\label{lem1}
            Suppose that  $u \in BV ^{\vec p} ( \R^N)$, and that $|u|_{ p^\star} \leq 1$, then 
             $$ {\cal K}  |u|_{p^\star}^{ p^+} \leq  | \nabla_1 u |_1 + \sum_{i=N_1+1}^N {1 \over p_i}  | \partial_i u |_{p_i}^{ p_i}$$
           and analogously if $u \in { \cal D}^{1, \vec p_\epsilon}( \R^N)$,  $|u|_{ p^\star_\epsilon} \leq 1 $
                          $$ {\cal K} _\epsilon |u|_{p_\epsilon^\star}^{ p_\epsilon^+} \leq {1 \over 1+ \epsilon}  | \nabla_1 u |_{1+\epsilon}^{1+ \epsilon}  + \sum_{i=N_1+1}^N {1 \over p_i^\epsilon }| \partial_i u |_{p_i^\epsilon}^{ p_i^\epsilon}. $$ 
                           \end{lemme}
                           
                            Hint of the proof : 
                            
                Use  ${ u \over |u|_{ p^\star}}$ in the definition of ${\cal K} $ and the fact that if $|u|_{ p^\star} \leq 1$,    $|u|_{ p^\star} ^{ p_i} \geq |u|_{ p^\star} ^{p^+}$.

             \begin{prop}\label{limsup}
              One has 
              $$\limsup {\cal K} _\epsilon \leq {\cal K} .$$
              As a consequence  any  sequence $(v_\epsilon)$  of extremal functions for $ {\cal K}_\epsilon$ is bounded independently on $\epsilon$, more precisely  there exists some positive constant $c$ so that, for all $\epsilon>0$, 
               $ | \nabla_1 v_\epsilon|_{1+ \epsilon}, | \partial_i v_\epsilon |_{ p_i^\epsilon} \leq c , \ { \rm for\ all}\ i\geq N_1+1$.
               
               \end{prop}
              
               \begin{proof}
Let $\delta >0$, $\delta < {1\over 2}$ and let $u_\delta \in \mathcal{D}^{1,p}(\R^N)$, ( or $BV^{\vec p} ( \R^N)$), so that $|u_\delta|_{p^*}=1$  and 
$$ | \nabla_1 u_\delta|_1 + \sum_{ i= N_1+1}^N {1 \over p_i}  | \partial_i u_\delta|_{ p_i} ^{p_i} \leq {\cal K}  + \delta.$$

By definition of $\mathcal{D}^{1,\vec p}(\R^N)$, there exists $v_\delta \in \mathcal{D}(\R^N)$ such that  $||v_\delta|_{ p^\star } -1| \leq \delta$, 
$$| \nabla_1 v_\delta|_1 + \sum_{i= N_1+1}^N {1 \over p_i} | \partial_i v_\delta|_{ p_i} ^{p_i} \leq {\cal K} + 2\delta.$$
For $\epsilon$ small enough one has $||v_\delta|_{ { p_\epsilon^\star} } -1| \leq 2\delta$.  
By considering $w_\delta^\epsilon = {v_\delta \over |v_\delta|_{p_\epsilon^\star}}$,  one  sees that $w_\delta^\epsilon  \in \mathcal{D}(\R^N)$, $|w_\delta^\epsilon |_{ p_\epsilon ^\star} = 1$,  
 and 
 \begin{eqnarray*}
 | \nabla_1 w_\delta^\epsilon|_1 +\sum_{i=N_1+1}^N{1 \over p_i}  | \partial_i w_\delta^\epsilon|_{ p_i} ^{p_i} &\leq  &{| \nabla_1 v_\delta|_1 \over 1-2\delta} +\sum_{i=N_1+1}^N{1 \over p_i}{ | \partial_i v_\delta|_{ p_i} ^{p_i}\over (1-2\delta)^{p_i}} \\
 &\leq & {1 \over (1-2\delta)^{p^+}} ({\cal K} +2 \delta). 
 \end{eqnarray*}
 By  the Lebesgue's dominated convergence theorem,  $| \nabla_1 w_\delta^\epsilon|^{1+ \epsilon} _{1+ \epsilon} \rightarrow {| \nabla_1 v_\delta|_1\over |v_\delta|_{p^*}}$, and 
$ | \partial_i w_\delta^\epsilon|_{p_i^\epsilon}^{p_i^\epsilon} \rightarrow {| \partial_i v_\delta|_{p_i}^{p_i}\over |v_\delta|_{p^*}^{p_i}}$  for all $i \geq N_1+1$ when $\epsilon \to 0$, hence we get  $$\underset{\epsilon\to 0}{{\limsup}} {\cal K}_\epsilon \leq \underset{\epsilon\to 0}{{\limsup}}\left[{1 \over 1+\epsilon}  | \nabla_1 w_\delta^\epsilon|_{1+\epsilon}^{ 1+\epsilon} + \sum_{i=N_1+1}^N{1 \over p_i^\epsilon}| \partial_i w_\delta^\epsilon|_{ p_i^\epsilon } ^{p_i^\epsilon } \right] \leq {{\cal K} + 2\delta\over (1-2\delta)^{p^+}},$$
which concludes the proof since $\delta$ is arbitrary.

               \end{proof}

                 \begin{prop}\label{propbl}
                  Suppose that $w_\epsilon \in BV^{ \vec p}( \R^N) $ satisfies $|w_\epsilon |_{ p^\star} = 1$, and that $w_\epsilon \rightarrow v$ almost everywhere. 
                   Then  for $\epsilon$ small enough 
                   $ |w_\epsilon -v| _{p^\star} \leq 1$.
                    \end{prop}
                     \begin{proof} 
                      If $v\equiv  0$, there is nothing to prove. If  $v\neq 0$, using Brezis Lieb Lemma, \cite{BL} one has 
                    $ |w_\epsilon -v| _{p^\star} -(  |w_\epsilon|_ {p^\star} -|v|_{ p^\star}) \rightarrow 0$
          which implies that $\limsup |w_\epsilon -v|_ {p^\star} <1$, hence the result holds. 
           This lemma will be used for $w_\epsilon = v_\epsilon ^{ \lambda_\epsilon} $ , where $v_\epsilon$ is some convenient extremal function,  given in Lemma \ref{lemveps} below, and $\lambda_\epsilon$ has been defined in ( \ref{eqlambda}). 
           \end{proof}

     \begin{lemme}\label{lemveps}
     Let $u_\epsilon$ be  a non negative extremal function for ${\cal K}_\epsilon$, so that $|u_\epsilon |_{ p_\epsilon ^\star} = 1$. There exists $v_\epsilon\geq 0$  which satisfies 
    \begin{eqnarray*}
          |u_\epsilon|_{p^*_\epsilon}&=& |v_\epsilon|_{p^*_\epsilon}=1,~
           |\nabla_1 u_\epsilon|_{1+\epsilon} = |\nabla_1 v_\epsilon|_{1+\epsilon}, \text{ and }
          |\partial_i u_\epsilon|_{p_i^\epsilon}= |\partial_i v_\epsilon|_{p_i^\epsilon}, \text{ for all } i \geq N_1+1, \nonumber\\
           &&\text{ and } \int_{ B(0, 1)} v_\epsilon ^{ p_\epsilon^\star} = {1 \over 2}  .  
     \end{eqnarray*}
 \end{lemme}

     \begin{proof}

      This proof is as in \cite{HR2}, but we reproduce it here for the reader's convenience. 
       Let $\alpha_i^\epsilon={p^*_\epsilon\over p_i^\epsilon}-1, i=1,\cdots,N$. For every $y=(y_1,\cdots,y_N)\in \R^N$,  and for any $u\in \mathcal{D}^{1,\vec{p_\epsilon}}(\R^N)$, and $t>0$, we set 
$$u^{t,y}(x)= t u(t^{\alpha_1^\epsilon}(x_1-y_1),\cdots, t^{\alpha_N^\epsilon} (x_N-y_N)).$$
Then, we have
$$|u|_{p^*_\epsilon}= |u^{t,y}|_{p^*_\epsilon},$$
$$|\partial_i u|_{p_i^\epsilon}= |\partial_i u^{t,y}|_{p_i^\epsilon}, \text{ for all }1\leq i \leq N,$$
$$|\nabla_1 u|_{1+\epsilon} = |\nabla_1 u^{t,y}|_{1+\epsilon}.$$
 
Let $u_\epsilon$ be an extremal  function for ${\cal K}_\epsilon$ so that $|u_\epsilon|_{ p_\epsilon^\star} = 1$. 
As in \cite{HR2}, \cite{PL2}, we recall  the  definition of the Levy concentration function, for $t>0 :$
$$Q_\epsilon(t) = \underset{y\in \R^N}{\sup} \int_{E(y,t^{\alpha_1^\epsilon},\cdots,t^{\alpha_N^\epsilon})} |u_\epsilon|^{p^*_\epsilon},$$
where $E(y,t^{\alpha_1^\epsilon},\cdots,t^{\alpha_N^\epsilon})$ is the ellipse defined by
$$\{z=(z_1,\cdots,z_N)\in \R^N, \sum_{i=1}^N {(z_i-y_i)^2\over t^{2\alpha_i^\epsilon}} \leq 1 \},$$
with $y=(y_1,\cdots,y_N)$, and $\alpha_i^\epsilon={p^*_\epsilon \over p_i^\epsilon}-1$ for all $i$. 
     
     Since for every $\epsilon>0$, $\underset{t\to 0}{\lim} Q_\epsilon(t)=0$, and $\underset{t \to \infty}{\lim} Q_\epsilon(t)= 1$, there exists $t_\epsilon >0$ such that $Q_\epsilon(t_\epsilon)={1\over 2}$, and  there exists $y_\epsilon\in \R^N$ such that
     $$ \int_{E(y_\epsilon,t_\epsilon^{\alpha_1^\epsilon},\cdots,t_\epsilon^{\alpha_N^\epsilon})} |u_\epsilon|^{p^*_\epsilon}(x) dx = {1\over 2}.$$
     
     Thus, by a change of variable one has for $v_\epsilon= u_\epsilon^{t_\epsilon, y_\epsilon}$ : 
     $$\int_{B(0,1)} |v_\epsilon|^{p^*_\epsilon}= {1\over 2} = \underset{y\in \R^N}{\sup} \int_{B(y,1)} |v_\epsilon|^{p_\epsilon^*}.$$ 
     
      Note for further purpose that $v_\epsilon$ is also extremal for ${\cal K}_\epsilon$.  
      
            \end{proof}
            
        \begin{prop}
        Let  $v_\epsilon\geq 0$  be  in ${ \cal D}^{1, \vec p_\epsilon}( \R^N)$, bounded in that space, independently on $\epsilon$.   Then for  $\lambda_\epsilon$ defined in (\ref{eqlambda}), the sequence 
         $w_\epsilon = v_\epsilon^{ \lambda_\epsilon}$ is bounded in ${ \cal D}^{1, \vec p}( \R^N)$. 
          \end{prop}
          
           \begin{proof}
           One has 
           \begin{eqnarray*}
           \int _{\R^N} | \nabla_1 (v_\epsilon^{ \lambda_\epsilon})|&=&\lambda_\epsilon  \int_{\R^N}  v_\epsilon^{ \lambda_\epsilon-1}| \nabla_1 v_\epsilon|\\
           &\leq &\lambda_\epsilon(\int_{\R^N}  | \nabla_1 v_\epsilon|^{1+ \epsilon} )^{1 \over 1+ \epsilon} ( \int_{\R^N}  v_\epsilon^{ (\lambda_\epsilon-1)(1+ \epsilon) \over \epsilon})^{\epsilon \over 1+ \epsilon} \\
           &=&\lambda_\epsilon (\int _{\R^N}  |\nabla_1 v_\epsilon|^{1+ \epsilon} )^{1 \over 1+ \epsilon} ( \int_{\R^N}  v_\epsilon^{ p_\epsilon^\star})^{\epsilon \over 1+ \epsilon} 
           \end{eqnarray*}
       and 
                    for all $i > N_1$, using the definition  in (\ref{defai})
             \begin{eqnarray*}
           \int_{\R^N}   | \partial_i (v_\epsilon^{ \lambda_\epsilon})|^{p_i} &=& \lambda_\epsilon^{p_i}\int_{\R^N}  v_\epsilon^{( \lambda_\epsilon-1)p_i}| \partial_i  v_\epsilon|^{p_i} \\
           &\leq &\lambda_\epsilon^{p_i}(\int_{\R^N} | \partial_i v_\epsilon|^{p_i^\epsilon})^{1 \over 1+ \epsilon_i} ( \int_{\R^N}  v_\epsilon^{ (\lambda_\epsilon-1)p_i^\epsilon \over \epsilon_i})^{\epsilon_i \over 1+ \epsilon_i} \\
           &=&\lambda_\epsilon^{p_i}  (\int_{\R^N} | \partial_i v_\epsilon|^{p_i^\epsilon} )^{1 \over 1+ \epsilon_i}  ( \int_{\R^N}  v_\epsilon^{ p_\epsilon^\star})^{\epsilon_i \over 1+ \epsilon_i} 
           \end{eqnarray*}
          Then $  \int_{\R^N}   | \nabla_1 (v_\epsilon^{ \lambda_\epsilon})|$ and $ \int_{\R^N}   | \partial_i (v_\epsilon^{ \lambda_\epsilon})|^{p_i} $ for $i \geq N_1+1$ are  bounded independently on $\epsilon$, by the assumptions.            
            \end{proof}

                      \bigskip

                 Let $v_\epsilon$ be given by Lemma \ref{lemveps}.
                                     One has by the definition of $\lambda_\epsilon$, 
                  $$\int_{\R^N}  |v_\epsilon |^{p_\epsilon^\star} = 1= \int_{\R^N}  |v_\epsilon ^{ \lambda_\epsilon} |^{ p^\star}.$$
                   Let us define 
                 $$
                    \lim_{ R \rightarrow +\infty} \limsup_{ \epsilon \rightarrow 0}  \int_{ |x|> R} |v_\epsilon^{\lambda_\epsilon}|^{ p^\star} = \nu_\infty,
                 $$
                      and 
                  $$
                       \lim_{ R \rightarrow +\infty} \limsup_{ \epsilon \rightarrow 0}  \int_{ |x|> R} \left( | \nabla_1 (v_\epsilon^{ \lambda_\epsilon }) | + \sum_{i=N_1+1}^N{1 \over p_i}| \partial_i (v_\epsilon^{ \lambda_\epsilon}) |^{p_i} \right)=  \mu_\infty
                      $$
                       while 
                         $$
\lim_{ R \rightarrow +\infty} \limsup_{ \epsilon \rightarrow 0} \int_{ |x|> R} \left( {1\over 1+ \epsilon} | \nabla_1 v_\epsilon |^{1+ \epsilon}  + \sum_{i=N_1+1}^N{1 \over p_i^\epsilon}| \partial_i v_\epsilon |^{p_i^\epsilon} \right)=\tilde  \mu_\infty, 
$$
 \begin{rema}\label{rema1} Note that since $\int_{ B(0, 1)} |v_\epsilon^{ \lambda_\epsilon} |^{ p ^\star} = {1\over 2}$,\ and $\int _{\R^N}  |v_\epsilon^{ \lambda_\epsilon} |^{ p ^\star} = 1$, \  $\nu_\infty \leq {1\over 2}$.
 \end{rema}
\begin{theo}\label{thcomp}
 Let $v_\epsilon \in { \cal D}^{1, \vec p_\epsilon}( \R^N) $,  be given by Lemma \ref{lemveps}, and  $\lambda_\epsilon$  be defined in (\ref{eqlambda}).  There exist positive  bounded measures on $\R^N$ :   $\tau, \tilde \tau,\mu^i,  \tilde \mu^i,$  for $  N_1+1 \leq i \leq N, $ and $\nu$,  a sequence of points $x_j\in \R^N$,   and some positif reals $ \nu_j, \mu_j^i ,   \tau_j, \tilde \tau_j, j \in \N$,  so that
  for a subsequence 
                              \begin{enumerate}
                   \item $v_\epsilon,$ and $v_\epsilon ^{ \lambda_\epsilon}$ converge both to $v$, almost everywhere and strongly   in every $L^q_{loc}$, $q < p^\star$, and $v \in BV^{ \vec p} ( \R^N)$. 
                   \item $| \nabla_1 ( v_\epsilon ^{ \lambda_\epsilon})| \rightharpoonup  | \nabla_1 v| +   \tau $,  $| \nabla_1  v_\epsilon  |^{1+ \epsilon}  \rightharpoonup  | \nabla_1 v| + \tilde  \tau $ , with $\tilde \tau \geq  \tau$,  in $M^1( \R^N)$ weakly.  
                   \item $| \partial_i v_\epsilon ^{ \lambda_\epsilon}|^{p_i} \rightharpoonup  | \partial_i v |^{p_i} + \mu^i $ for all $i \geq N_1+1$,  $| \partial_i   v_\epsilon  |^{p_i^\epsilon} \rightharpoonup  | \partial_i v |^{p_i} + \tilde \mu^i  $ with $\tilde \mu^i \geq  \mu^i $,   
in $M^1( \R^N)$ weakly. 
                  
                                      \item $ |v_\epsilon ^{ \lambda_\epsilon }|^{p^\star}=| v_\epsilon| ^{p_\epsilon^\star }\rightharpoonup |v|^{ p^\star} + \nu:=   |v|^{ p^\star} + \sum_j \nu_j \delta_{x_j}$  in $M^1( \R^N)$ weakly. 
                   
                   \item  One has  $\tau \geq \sum_j \tau_j \delta_{ x_j}$, $\mu^i \geq \sum_j \mu_j^i \delta_{ x_j}$,  for all $ i \geq N_1+1$,  and for any  $j\in \N$,  $\nu_j^{ p^+ \over p^\star } \leq {1 \over {\cal K} } (  \tau_j + \sum_i {1 \over p_i} {\mu_j^i}  )$, and  $ \nu_\infty ^{ p^+\over p^\star}\leq {1 \over {\cal K} }  \mu_\infty$. 
                   
                         \item  \begin{eqnarray*}  | \nabla_1 ( v_\epsilon ^{ \lambda_\epsilon})|_1 &+& \sum _{i=N_1+1}^N {1 \over p_i}    | \partial_i v_\epsilon ^{ \lambda_\epsilon}|^{p_i}_{p_i} \rightarrow 
                      \int_{ \R^N} | \nabla_ 1 v | + \sum_{i=N_1+1}^N {1 \over p_i} \int_{ \R^N}  | \partial_i v |^{p_i} \\
                      &+& \int_{ \R^N} \left[\tau + \sum_{i=N_1+1}^N{1 \over p_i}  \mu^i \right]+ \mu_\infty.
                       \end{eqnarray*}
                     \item \begin{eqnarray*}
                      {1\over 1+ \epsilon} | \nabla_1  v_\epsilon  |_{1+ \epsilon}^{1+ \epsilon} &+& \sum_{i=N_1+1}^N {1 \over p_i^\epsilon} | \partial_i   v_\epsilon  |_{ p_i^\epsilon} ^{p_i^\epsilon}\rightarrow   \int_{ \R^N}  | \nabla_ 1 v | + \sum_{i=N_1+1}^N {1 \over p_i} \int | \partial_i v |^{p_i} \\
                      &+& \int_{ \R^N} \left[ \tilde \tau 
                      + \sum_{i=N_1+1}^N{1 \over p_i} \tilde  \mu^i \right]+ \tilde \mu_\infty. 
                      \end{eqnarray*}
                       \item $$\int_{ \R^N}  |v_\epsilon| ^{ p_\epsilon ^\star} =1= \int_{ \R^N}   |v_\epsilon| ^{ \lambda_\epsilon p^\star} \rightarrow \int_{ \R^N}  |v|^{ p^\star} + \int_{ \R^N}  \nu+ \nu_\infty. $$
                               
             \end{enumerate}
              \end{theo}
              \begin{proof}
              1 The convergence of $v_\epsilon ^{ \lambda_\epsilon}$ is clear  by  using the compactness  of the embedding from $BV^{\vec p}$ in $L^q$ with $q < p^\star < p_\epsilon^\star$, on bounded sets of $\R^N$,   the analogous for $v_\epsilon$ is also true since $q < \liminf p_\epsilon^\star$.  
              
            Let us prove the existence of  $\tilde \tau,\tau, \mu^i ,  \tilde \mu^i,  N_1+1 \leq i \leq N,$ and $ \nu$.  Indeed 
                 one has  by extracting a subsequence  the existence of $\tilde\tau$, since we know that $| \nabla_1 v | \leq \liminf | \nabla_1 v_\epsilon |^{1+ \epsilon}$. The existence of $\tau$  is obtained  from the same arguments. Furthermore, by H\"older's inequaltiy 
                  \begin{eqnarray*}
                \int   | \nabla_1( v_\epsilon^{\lambda_\epsilon} ) |\varphi &\leq &\lambda_\epsilon ( \int | \nabla_1 v_\epsilon |^{1+ \epsilon}  \varphi)^{ 1 \over 1+ \epsilon }   ( \int v_\epsilon^{ p_\epsilon^\star}\varphi)^{ \epsilon \over 1+ \epsilon}.
                \end{eqnarray*}
                Letting $\epsilon$ go to zero, since $\lambda_\epsilon $ goes to $1$, one gets that $\tilde \tau \geq \tau$. 
                 We argue in the same manner  to prove the analogous results for $ |\partial_i (  v_\epsilon^{\lambda_\epsilon} )|^{ p_i}$ and  $| \partial_i   v_\epsilon  |^{p_i^\epsilon}$. The existence of $\nu$ is clear.

             We prove in the lines which follow that $\nu$ is purely atomic.  This is classical, but we reproduce the proof for the convenience of the reader.  
             Let 
            $$ \mu= 2| \nabla_1 v| + \tau + \sum_{i=N_1+1}^N{2^{p_i-1}  \over p_i}  (   \mu^i+2 | \partial_i v|^{p_i}) $$                                  
      
      {\bf Claim 1}     
  For all $\varphi \in { \cal C}_c ( \R^N)$, 
  
\begin{equation} \label{inmesure}
(\int | \varphi |^{p^\star} d\nu )^{1 \over p^\star} \leq (p^+)^{{1 \over N}+ {1 \over p^\star}} ( \int \mu)  ^{{1\over N}+ {1 \over p^\star}-{1\over p^+}} {1 \over T_o}( \int | \varphi |^{p^+}d \mu)^{1 \over p^+}
\end{equation}

To prove  {\bf Claim 1},   let us define  $h_\epsilon  =( v_\epsilon^{ \lambda_\epsilon}  -v) $.   Using  (\ref{eqemb}), 
     \begin{equation}\label{passingtothelim}
      ( \int |h_\epsilon  \varphi |^{p^\star } )^{1 \over p^\star} \leq {1 \over T_o}\Pi_1^N(  \int | \partial_i ( h_\epsilon  \varphi ) |^{p_i})^{1 \over N p_i} .
     \end{equation}

We  have defined $\nu$ and $\mu^i $ by the following  vague convergences :   $v_\epsilon^{\lambda_\epsilon p^\star} \rightharpoonup v^{p^\star} + \nu $, $|\partial_i v_\epsilon^{\lambda_\epsilon}|^{p_i} \rightharpoonup |\partial_i v |^{p_i} + \mu^i$, and $| \nabla_1 v_\epsilon ^{ \lambda_\epsilon}| \rightharpoonup | \nabla_1 v| + \tau$. By  Bresis Lieb's  Lemma, one derives that 

$$|h_\epsilon|^{p^*} \rightharpoonup \nu, $$
 while 
 $$| \partial_i h_\epsilon |^{ p_i} \leq 2^{p_i-1} ( | \partial_i v_\epsilon ^{ \lambda_\epsilon}|^{p_i} + | \partial_i v |^{p_i})\rightharpoonup 2^{p_i-1} ( 2  | \partial_i v |^{p_i}+ \mu^i), $$
 and 

$$|\nabla_1 h_\epsilon | \leq |\nabla_1 (v_\epsilon)^{\lambda_\epsilon}| + |\nabla_1 v | \rightharpoonup 2 |\nabla_1 v| + \tau \text{ vaguely}.$$
     Using  the fact that $h_\epsilon $ tends to $0$ in $L^{p_i} ( Suppt \varphi)$, for all $i$,  since $p_i < p^\star$, one has $\int |h_\epsilon |^{p_i
} | \partial_i \varphi |^{p_i} \rightarrow 0$. 
Passing to the limit in \eqref{passingtothelim}, one gets 
     $$ \left(\int | \varphi |^{p^\star } d \nu\right)^{1\over p^\star} \leq {1 \over T_o} \left(\int | \varphi |d ( 2| \nabla v| + \tau)\right) ^{N_1\over N}  \Pi_{ i =N_1+1}^N \left( \int |  \varphi  |^{p_i} d (2^{p_i-1} ( 2  | \partial_i v |^{p_i}+ \mu^i ) ) \right)^{1 \over N p_i} .$$ 
        We then use  for $i \geq N_1+1$
      $$ \int |  \varphi  |^{p_i} d (2^{p_i-1} ( 2  | \partial_i v |^{p_i}+ \mu^i ) )\leq p^+( \int\mu )^{1-{p_i \over p^+} } (\int | \varphi |^{p^+}d \mu )^{{p_i \over p^+}},  $$
       and 
       $$\int | \varphi | d ( 2| \nabla v| + \tau)\leq  p^+ (\int \mu) ^{1-{1 \over p^+}} (\int | \varphi |^{p^+}d \mu )^{{1 \over p^+}}. $$
      Taking the power ${1\over Np_i}$ and ${N_1 \over N} $ and multiplying the inequalities,  one derives {\bf  Claim  1}. 
      
      \bigskip

          By  \eqref{inmesure}    one sees that $\nu$ is absolutely continuous with respect to $\mu$, with for some constant $c$ and 
          for any borelian set $E$, 
           $$\nu (E) \leq c \mu (E)^{ p^\star \over p^+}$$

           Let then  $h\geq 0$ be  $\mu$ integrable so that $\nu = h d \mu$. Then if $x$ is a density point for $\mu$, ie, so that 
            $\lim_{ r \rightarrow 0} \mu (B(x, r))=0$, one gets that 
            ${\nu (B(x, r)) \over \mu ( B(x, r))}  \rightarrow 0$, hence  if $D$ is the at most numerable set where $\mu ( \{ x_j\}) >0$, one has 
            $h= 0$ in $\R^N \setminus D$. This implies that $\nu$ has only atoms that we will denote $\{x_j\}_{j\in \N} $. 
            
            \medskip
                
    We now prove     5. 
    We still follow the lines in \cite{HR2}.
    
     Let $\delta >0$ small, $q_i = { p_i p^\star \over p^\star-p_i}$, $\alpha_i = {1\over q_i}$, ( note that $\sum_{i=1}^N \alpha_i = 1$), define  for $j\in \N$ fixe,   $\phi \in { \cal D} (B(0,1))$, $\phi (0)= 1$, $0\leq \phi\leq 1$
    the function $\phi_\delta$ as    $ \phi_\delta ( x)=  \phi ( {x-x_j^1 \over \delta^{ \alpha_1}}, \cdots , {x-x_j^N \over  \delta^{ \alpha_N}})$. 
     $\phi_\delta$ satisfies 
      $\int _{\R^N} | \partial_i \phi_\delta |^{q_i} = \int_{\R^N}  | \partial_i  \phi|^{ q_i}$. 
      In particular  for all $i \leq N$, 
       \begin{equation}\label{phidelta}\int_{\R^N}  | \partial_i \phi_\delta|^{ p_i} v^{ p_i} \leq ( \int_{\R^N}  | \partial_i \phi_\delta |^{q_i} )^{p_i\over q_i}( \int_{ B(x_j, \max_i \delta  ^{ \alpha_i})} v^{ p^\star})^{ p_i \over p^\star} \rightarrow 0, 
       \end{equation}
        when $\delta$ goes to zero. 
       
       {\bf Claim 2}
     $$
        {\cal K}   \nu_j^{ p^+\over p^\star} \leq \limsup_{ \delta\rightarrow 0} \limsup_{ \epsilon \rightarrow 0} \ \int_{\R^N} \left(  \phi_\delta | \nabla_1  v_\epsilon ^{ \lambda_\epsilon } | + \sum_{ i = N_1+1}^N {1\over p_i} | \partial_i ( v_\epsilon^{ \lambda_\epsilon})|^{p_i} \phi_\delta^{ p_i}\right)
     $$
        
       To prove  {\bf Claim 2},  we apply  Lemma \ref{lem1}  with  $|v_\epsilon ^{ \lambda_\epsilon} \phi_\delta  |_{ p^\star} \leq 1$
         $${\cal K} (\int_{\R^N}  |v_\epsilon ^{ \lambda_\epsilon} \phi_\delta  |^{ p^\star})^{ p^+ \over p^\star}  \leq \int_{\R^N}  | \nabla_1 ( v_\epsilon ^{ \lambda_\epsilon} \phi_\delta)| + \sum_{i=N_1+1}^N{1 \over p_i} \int_{\R^N}  | \partial_i ( v_\epsilon ^{ \lambda_\epsilon} \phi_\delta )|^{p_i}. $$
           We use 
           \begin{eqnarray*}
           \left\vert  |\nabla_1 ( v_\epsilon ^{ \lambda_\epsilon} \phi_\delta)| -| \nabla_1 ( v_\epsilon ^{ \lambda_\epsilon}) | \phi_\delta \right\vert &\leq & v_\epsilon ^{ \lambda_\epsilon}  | \nabla_1 \phi_\delta|\\
            &\leq & | v_\epsilon ^{ \lambda_\epsilon} -v| | \nabla_1  \phi_\delta| + v| \nabla_1 \phi_\delta|, 
                       \end{eqnarray*}
                   hence by    (\ref{phidelta}) when $p_i = 1$ and  $  v_\epsilon^{\lambda_\epsilon}-v\rightarrow 0$ in  $L^q_{loc}$ for all $q < p^\star$, this goes to zero in  $L^1$ when $\epsilon$ and $\delta$ go to zero. 
           For $i \geq N_1+1$,  by the mean value's theorem 
             \begin{eqnarray*}
            \left\vert  | \partial_i ( v_\epsilon ^{ \lambda_\epsilon} \phi_\delta)|^{p_i}\right.&&\left. - | \partial_i ( v_\epsilon ^{ \lambda_\epsilon}) \phi_\delta|^{p_i} \right\vert \\
            &\leq & 
            p_i | (\partial_i \phi_\delta )  v_\epsilon ^{ \lambda_\epsilon} | \ \left\vert  | \partial_i \phi_\delta | v_\epsilon ^{ \lambda_\epsilon} + |  \partial_i ( v_\epsilon ^{ \lambda_\epsilon} )\phi_\delta|\right\vert^{p_i-1} \\
            &\leq & p_i  \left(  |\partial_i \phi_\delta  || v_\epsilon ^{ \lambda_\epsilon} -v|+  |\partial_i \phi_\delta |  v \right) \left\vert |(  \partial_i \phi_\delta)  v_\epsilon ^{ \lambda_\epsilon}| +  |  \partial_i ( v_\epsilon ^{ \lambda_\epsilon} )|\phi_\delta\right\vert^{p_i-1}.
            \end{eqnarray*}
           Using  Holder's inequality, ( \ref{phidelta}) for $i \geq N_1+1$,    the fact that  $\left\vert  |  \partial_i \phi_\delta | v_\epsilon ^{ \lambda_\epsilon} +  |  \partial_i ( v_\epsilon ^{ \lambda_\epsilon} )\phi_\delta|\right\vert^{p_i-1}
$ is bounded in $L^{ p_i\over p_i-1}$,  and $  v_\epsilon^{\lambda_\epsilon}-v\rightarrow 0$ in  $L^q_{loc}$ for all $q < p^\star$, this goes to zero in $L^1$, when $\epsilon$ and $\delta$ go to zero.
{\bf Claim 2} is proved.

\medskip

  We can now conclude, using the fact that $|\nabla_1v|$ is orthogonal to Dirac masses,  as a  consequence of the results on the dimension of the  support of $| \nabla_1 v|^s$, 
 \cite{Gi}, and using  the fact that $| \partial_i v |^{p_i}$  belongs to $L^1$,  for $i\geq N_1+1$,  that 
   $${\cal K}  \nu_j^{ p^+\over p^\star} \leq \limsup _{ \delta\rightarrow 0} (\int_{\R^N}  \tau \phi_\delta + \sum_{i=N_1+1}^N{1\over p_i} \int_{\R^N}  \mu^i\phi_\delta^{p_i} )$$
   Defining $\tau_j = \limsup_{ \delta \rightarrow 0} \int_{\R^N}  \tau \phi_\delta$ and 
    $\mu_j ^i = \limsup_{ \delta \rightarrow 0} \int_{\R^N} \mu^i \phi_\delta^{p_i}$, one gets the first part of 5.

          To prove the last part of 5,             let $R>0$ large  and $\psi_R$ some ${ \cal C}^\infty$ function which is $0$ on $|x|<R$, and equals    $1$  for $|x| > R+1$, $0\leq \psi_R \leq 1$. 
         It can easily be seen that for  any $i\geq N_1+1$ and for any $\gamma_i \geq 1$
                    \begin{equation}\label{eqmuinf1} \int_{ |x| > R+1}  | \partial_i v_\epsilon ^{ \lambda_\epsilon} |^{p_i} \leq  \int_{ \R^N}  | \partial_i v_\epsilon ^{ \lambda_\epsilon} |^{p_i} \psi_R^{ \gamma_i}\leq   \int_{ |x| > R}  | \partial_i v_\epsilon ^{ \lambda_\epsilon} |^{p_i}
                    \end{equation}
                     
                     \begin{equation}\label{eqmuinf2}   \int_{ |x| > R+1}  | \nabla_1 v_\epsilon^{ \lambda_\epsilon} |\leq  \int_{ \R^N}  | \nabla_1v_\epsilon ^{ \lambda_\epsilon} | \psi_R^{\gamma_1}\leq  \int_{ |x| > R}  | \nabla_1 v_\epsilon^{ \lambda_\epsilon} |
                     \end{equation}
                      and 
                        \begin{equation}\label{eqnuinf}
                          \int_{ |x| > R+1}  | v_\epsilon ^{ \lambda_\epsilon} |^{p^\star} \leq  \int _{\R^N} | v_\epsilon ^{ \lambda_\epsilon} |^{p^\star}\psi_R^{p^\star} \leq  \int_{ |x| >R }  | v_\epsilon ^{ \lambda_\epsilon} |^{p^\star}.
                          \end{equation}
                                                               And then by the definition of $\mu_\infty$
                   
                   \begin{equation*}
                   \underset{R\to+\infty}{\lim} \underset{\epsilon\to 0}{\limsup}\int_{\R^N} |\nabla_1 v_\epsilon^{\lambda_\epsilon}| \psi_R + \sum_{i=N_1+1}^N {1\over p_i} \int_{\R^N} |\partial_i v_\epsilon^{\lambda_\epsilon}|^{p_i}\psi_R^{p_i}  = \mu_\infty. 
                   \end{equation*}
Let us remark that  since $v\in BV^{\vec p}$,  one has 
  $\lim_{ R \rightarrow +\infty} \int | \nabla_1 v | \psi_R +\sum_{i = N_1+1}^N {1\over p_i} \int | \partial_i v |^{p_i} \psi_R^{ p_i}  + \int_{\R^N} |v|^{ p^\star} \psi_R^{ p^\star} = 0$. 
          
                       We use  once more $h_\epsilon  = v_\epsilon ^{ \lambda_\epsilon}-v$,  which  goes to zero in $L^q_{loc}$.              
                Note that since $|h_\epsilon |_{ p^\star} \leq 1$, one also has $|h_\epsilon  \psi_R |_{ p^\star} \leq 1$ and then applying   Lemma \ref{lem1}
               \begin{equation}\label{pourlineginfini}
               {\cal K}  (\int |h_\epsilon  \psi_R|^{ p^\star} )^{ p^+ \over p^\star} \leq \int | \nabla_1( h_\epsilon  \psi_R) | + \sum_{ i = N_1+1}^N {1 \over p_i} \int | \partial_i ( h_\epsilon  \psi_R)|^{p_i}.
               \end{equation}
                                Since $\nabla \psi_R$ is compactly supported  in $R<|x| < R+1$,  and since $p_i < p^\star$
                  one has 
                  $$\underset{\epsilon\to0}{\lim}  \int_{\R^N}  h_\epsilon | \nabla_1 ( \psi_R) |+ \sum_{ i = N_1+1}^N {1 \over p_i} \int _{\R^N}| \partial_i \psi_R |^{p_i} h_\epsilon  ^{ p_i} = 0.$$
                                    Then 
                $$\lim_{ R \rightarrow +\infty} \limsup_{\epsilon\rightarrow 0}       \int_{\R^N} | \nabla_1( h_\epsilon  \psi_R) | + \sum_{ i = N_1+1}^N {1 \over p_i} \int_{\R^N} | \partial_i ( h_\epsilon  \psi_R)|^{p_i} = \mu_\infty.
              $$
                Note  also that 
                    $ \underset{R\to +\infty}{\lim} \underset{\epsilon\to0}{\limsup}\  {\cal K}  (\int_{\R^N} |h_\epsilon  \psi_R|^{ p^\star} )^{ p^+ \over p^\star} =  {\cal K}  \nu_\infty ^{ p^+ \over p^\star}$,                   
                                           hence, taking the limit in \eqref{pourlineginfini}, one gets 
                                         $ {\cal K}  \nu_\infty ^{ p^+ \over p^\star}\leq \mu_\infty$.

                   To show 6.       by the definition of $\tau$ and $\mu^i$,                               \begin{multline*}
                       \underset{R\to+\infty}{\lim} \underset{\epsilon\to 0}{\limsup}\int_{\R^N} |\nabla_1 v_\epsilon^{\lambda_\epsilon}| (1-\psi_R) + \sum_{i=N_1+1}^N {1\over p_i}\int_{\R^N} |\partial_i v_\epsilon^{\lambda_\epsilon}|^{p_i}(1-\psi_R)\\= \int_{\R^N} |\nabla_1 v| + \int_{\R^N} \tau + \sum_{i=N_1+1}^N {1\over p_i}  \int_{\R^N} (|\partial_i v|^{p_i} + \mu^i ).
                   \end{multline*}
                                 And then one gets 6. by writing $1= \psi_R + (1-\psi_R)$  and using (\ref{eqmuinf1}) and (\ref{eqmuinf2}). 
                   
                  7 can be proved in the same manner. 8 is obtained by gathering 4.  and (\ref{eqnuinf}).

\end{proof}

             \begin{proof}{ of Theorem \ref{exiext}}
             We take  a subsequence  $v_{\epsilon^\prime}$ so that 
             $${1\over 1+ \epsilon^\prime }\int_{\R^N} | \nabla_1
 v_{\epsilon ^\prime}|^{1+ \epsilon^\prime }   + \sum_{i= N_1+1}^N  {1\over p_i^{ \epsilon^\prime} } \int  | \partial_i v_{\epsilon^\prime} |^{p_i^{\epsilon^\prime}} ={\cal  K}_{\epsilon^\prime} $$ with $\lim    {\cal K} _{\epsilon^\prime}= \liminf {\cal K} _\epsilon,$        
in the sequel we will   still denote it  $v_\epsilon$ for simplicity . 
             
                     We  are going to prove both that 
               $\limsup  {\cal K} _\epsilon = {\cal K} = \liminf {\cal K} _\epsilon $,    $\nu_\infty = \mu_\infty = 0, \mu_j^i = \nu_j = 0$,  for all $j \in \N$,  that  for all $i$ $| \partial_i v_\epsilon |^{ p_i^\epsilon}\rightarrow | \partial_i v|^{p_i}$, tightly on $\R^N$,  and 
                that $\lim | \nabla_1( v_\epsilon ^{ \lambda_\epsilon})| = \lim | \nabla_1 v_\epsilon |^{1+ \epsilon} = | \nabla_1 v|$, tightly on $\R^N$. 
                                 Indeed, using the previous convergences in Theorem \ref{thcomp}
                  \begin{eqnarray*}
                  \int_{\R^N} | \nabla_1v | &+& \int_{\R^N}  \tau + \sum_{i=N_1+1}^N {1 \over p_i} \int_{\R^N} | \partial_i v |^{ p_i} + \sum_{i=N_1+1}^N{1\over p_i} \int_{\R^N} \mu^i + \mu_\infty \\
                  &\leq &  \int_{\R^N} | \nabla_ 1v | + \int_{\R^N} \tilde  \tau + \sum_{i=N_1+1}^N {1 \over p_i} \int_{\R^N} | \partial_i v |^{ p_i} + \sum_{i=N_1+1}^N{1\over p_i} \int_{\R^N} \tilde \mu^i +\tilde  \mu_\infty\\
                             &\leq   &   \lim  
                 {1\over 1+ \epsilon}   \int_{\R^N} | \nabla_1  v _\epsilon  |^{1 + \epsilon } + \sum_{i=N_1+1}^N{1 \over p_i^\epsilon} \int_{\R^N} | \partial_i (v_\epsilon )|^{ p_i^ \epsilon}
            \\
            &=& \liminf { \cal K}_\epsilon = \liminf { \cal K}_\epsilon ( |v|^{p^\star} + \sum \nu_j + \nu_\infty )^{ p^+ \over p^\star}
          \\                                       
            & \leq&   \liminf { \cal K}_\epsilon\left( ( |v|^{p^\star}_{p^\star} )^{p^+\over p^\star} + (\sum \nu_j )^{ p^+ \over p^\star}+ \nu_\infty^{p^+\over p^\star}  \right)
             \\                                                              &\leq &\liminf {\cal K} _\epsilon \left( \int_{\R^N} | v|^{p^\star}_{p^\star} \right)^{ p^+ \over p^\star} + {\liminf {\cal K} _\epsilon \over {\cal K} }\left[ \sum_j  (\ \tau_j \right.
                                      + \left.\sum_{i= N_1+1}^N{ 1 \over p_i}  \mu_j ^i )+  \mu_\infty\right]  
                                       \end{eqnarray*}
                   \begin{eqnarray*}\\  
     \\                                          &\leq & {\liminf {\cal K} _\epsilon \over {\cal K} } \left(   \int_{\R^N} | \nabla_1v | +  \sum_{i= N_1+1}^N{1 \over p_i} \int_{\R^N} | \partial_i v |^{ p_i} \right)            \\                                        &+& {\liminf{\cal K} _\epsilon \over {\cal K} }\left( \sum_j \tau_j + \sum_j \sum_{i= N_1+1}^N{1 \over p_i} \mu_j ^i + \mu_\infty\right)\\
                                     \end{eqnarray*}
                    Using the fact that $ \limsup {\cal K} _\epsilon \leq {\cal K} $, $ \int_{ \R^N}  \tau \geq \sum_j   \tau_j$, $\int_{ \R^N}\mu^i  \geq  \sum_j   \mu_j ^i $, one gets that we have equalities in place of inequalities everywhere we used them. In particular 
                    $ ( \int_{\R^N} | v|^{p^\star} + \sum_j \nu_j + \nu_\infty)^{ p^+ \over p^\star}  = ( \int_{\R^N} | v|^{p^\star} )^{ p^+ \over p^\star}  +  \sum_j   \nu_j ^ {p^+ \over p^\star} + \nu_\infty ^ {p^+ \over p^\star} $ , and then only one of the positive  reals $\int_{\R^N} | v|^{p^\star} , \nu_j , \nu_\infty$,  can be  different from zero. But this imposes that the only one which is $\neq 0$ must be equal  to  one.  By Remark \ref{rema1}, one then  gets $\nu_\infty = 0$. On the other hand, let $j\in\N$,  either $x_j \notin B(0, 1)$  and then  for $\delta$ small enough $\int_{B(x_j, \delta)} |v_\epsilon |^{ p_\epsilon^\star} + \int _{ B(0, 1)} |v_\epsilon |^{ p_\epsilon^\star}  \leq 1$, hence $\nu_j = 0$, or $x_j \in B( 0,1) $ and then 
                     $\nu_j \leq  \lim\int_{ B(0, 1)} |v_\epsilon |^{ p_\epsilon^\star} = {1\over 2}$, and once more $\nu_j = 0$. \     One then derives that                          
                     $1=  |v_\epsilon|_{ p_\epsilon^\star} ^{ p_\epsilon^\star} \rightarrow |v|^{p^\star}_{p^\star}$.   By the definition of ${\cal K}$ one has 
            \begin{eqnarray*}
              {\cal K}  \leq    | \nabla_1 v |_1 + \sum_{N_1+1}^N {1 \over p_i } | \partial_i v |^{ p_i} & \leq&   | \nabla_1 v |_1 + \sum_{N_1+1}^N {1 \over p_i } | \partial_i v |^{ p_i} +\tilde \tau+  \sum_{N_1+1}^N{1 \over p_i} \tilde \mu^i + \tilde \mu_\infty \\
              & \leq& \liminf {\cal K} _\epsilon \leq \limsup {\cal K} _\epsilon \leq  {\cal K} 
              \end{eqnarray*}
                         and then  $\tilde \tau= \tau =\tilde  \mu_\infty= \mu_\infty  =\tilde  \mu^i =\mu^i= 0$,  
             $\lim | \nabla _1 v_\epsilon|_{1+ \epsilon}^{1+ \epsilon} = \lim | \nabla_1( v_\epsilon ^{ \lambda_\epsilon})|_1= | \nabla_1 v|_1$, and for all $i \geq N_1+1$, both  $ |\partial_i ( v_\epsilon^{ \lambda_\epsilon} )|_{p_i} ^{ p_i}$ and $ | \partial_i v_\epsilon |_{ p_i^\epsilon} ^{ p_i^\epsilon}$ converge to $| \partial_i v|^{p_i}_{p_i}$.   
             We have obtained that  $v$ is an extremal function, and     $\lim {\cal K} _\epsilon = {\cal K}$.          
             \bigskip
             
              We now prove that $v$ satisfies (\ref{partial differential equation }). 
               First recall that $l_\epsilon\geq {\cal K} _\epsilon \geq {1 \over p^+} l_\epsilon$, as we can see by multiplying (\ref{extremalepsilon}) 
              by $v_\epsilon$ the equation, integrating,  and using $|v_\epsilon |_{ p_\epsilon ^\star} ^{p_\epsilon^\star} = 1$.  In particular $l_\epsilon$ is bounded. Let us extract from it a subsequence which converges to some $l\geq 0$.

               Let us define 
                                           $\sigma^{1, \epsilon} = |\nabla_1 v_\epsilon|^{\epsilon-1} \nabla_1 v_\epsilon$,   
                                                           $\sigma_i^\epsilon = | \partial_i v_\epsilon |^{p_i ^\epsilon -2} \partial_i    v_\epsilon$ for $i \geq N_1+1$, and  -with an obvious abuse of notation- $\sigma_\epsilon= (\sigma^{1,\epsilon}, \sigma_{N_1+1}^\epsilon,\cdots,\sigma_N^\epsilon)$.         Note that            $\sigma^{1,\epsilon} $ is bounded in $L^q_{loc}$, for any $ q< \infty$. Indeed, let $K$ be a compact set, one has by Holder's inequality
                                     $\int_K  |\sigma^{1,\epsilon}| ^q  = \int _K | \nabla_1 v_\epsilon |^{ \epsilon  q} \leq (\int _K | \nabla_1  v_\epsilon |^{ 1+ \epsilon })^{ q \epsilon\over 1+ \epsilon} | K|^{ 1-{q\epsilon \over 1+ \epsilon }}$
               and then 
                $( \int _K|\sigma^{1,\epsilon} |^ q  )^{1\over q} \leq    ((1+ \epsilon)  {\cal K}_\epsilon)^{ \epsilon \over 1+ \epsilon} |K|^{{1\over q}-{\epsilon \over 1+ \epsilon}}$.
                Using the  boundedness of ${\cal K}_\epsilon$ one gets 
                 that $\sigma^{1,\epsilon}$ is bounded in $L^q_{loc}$,  hence converges up to subsequence  weakly in $L^q_{loc}$ to some $\sigma ^1$ which satisfies for any  compact  set $K$ 
                 $|{\sigma^1}|_{L^q( K)} \leq |K|^{1\over q},$ hence  ${\sigma^1} \in L^\infty( \R^N, \R^{N_1})$ and $| {\sigma^1}|_\infty \leq 1$.              Furthermore,   the strong convergence of $|\partial_i v_\epsilon|^{p_i^\epsilon} $ towards $|\partial_i v|^{p_i} $ in $L^1$ when $i \geq N_1+1$ ensures that  $ \sigma_i =    | \partial_i v |^{p_i-2} \partial_i v$. 
                From  these convergences, one gets  that defining  $\sigma=(\sigma^1,\sigma_{N_1+1},\cdots,\sigma_N)$,  by the definition in Theorem \ref{sigma1grad1},  
                $\sigma^\epsilon \cdot \nabla v_\epsilon$  converges to $\sigma \cdot \nabla v$ in the distribution sense.  Using $\sum_{ N_1+1}^N \sigma_i^\epsilon \partial_i v_\epsilon \rightarrow \sum_{N_1+1}^N \sigma_i \partial_i v$ in $L^1_{loc}$, one derives   that 
                $\sigma^{1,\epsilon}\cdot \nabla_1 v_\epsilon$  converges to $\sigma^1 \cdot \nabla_1 v$ in ${ \cal D}^\prime ( \R^N)$. 
                Since  $\sigma^{1,\epsilon}\cdot \nabla_1 v_\epsilon$  is also  bounded in $L^1$, this  convergence is    in fact  vague. 
                 By lower semi-continuity for the vague topology,   for any $\varphi \geq 0$ in ${ \cal C}_c ( \R^N)$
                  $$\int | \nabla_1 v | \varphi  \leq \underset{\epsilon \to 0}{\liminf} \int | \nabla_1 v_\epsilon |^{1+ \epsilon} \varphi = \underset{\epsilon \to 0}{\liminf} \int \sigma^{1, \epsilon} \cdot \nabla_1 v_\epsilon \varphi  = \langle  \sigma^1 \cdot \nabla_1 v,  \varphi\rangle$$ 
                  This implies that 
                  $| \nabla_1 v | \leq    \sigma^1 \cdot \nabla_1 v $ in the sense of measures, and since one always has the reverse inequality, we have obtained that ${\sigma^1} \cdot \nabla_1 v = | \nabla_1 v|$.                  
                   
                    We get by passing to the limit in \eqref{extremalepsilon} that $v$ satisfies the partial differential equation : 
                    
                    $$ -{\rm div}_1 ( {\sigma^1}) - \sum_{ i = N_1+1}^N\partial_i ( |\partial_i v |^{ p_i-2} \partial_i v )=l v^{ p^\star-1}
                    $$
                     with 
                    $$ \int_{\R^N} v^{ p^\star}= 1, \ {\rm  and} \ {\sigma^1} \cdot \nabla_1 v = | \nabla_1 v|.
                 $$ 
                     
                       Furthermore,  multiplying the equation by $v$ and integrating, one gets $l \geq {\cal K} >0$.
                    
       \end{proof}
       
         \subsection{Proof of  Theorem \ref{Linfty}}
    
          We will prove the $L^\infty$ regularity when $u$ is some extremal function which satisfies  (\ref{partial differential equation }), with  $l = 1$. Indeed one has 
           \begin{lemme}\label{lemuueps}
           Let $v_\epsilon$ and $v$ be as in Theorem \ref{exiext}. Then 
           $$u ( x) = v(  l^{-1} x_1, \cdots , l^{-1}x_{N_1}, l^{-{1\over p_{N_1+1}} }x_{N_1+1}, \cdots , l^{-1 \over p_N} x_N)$$
            and   
            
            $$u_\epsilon  ( x) = v_\epsilon (  l_\epsilon^{-1\over 1+\epsilon} x_1, \cdots , l_\epsilon^{-1\over 1+\epsilon}x_{N_1}, l_\epsilon^{-{1\over p^\epsilon_{N_1+1}}} x_{N_1+1}, \cdots ,  l_\epsilon ^{-1 \over p^\epsilon_N} x_N)$$ 
                         satisfy respectively 
             $$ -{\rm div}_1 ( {\sigma^1(u)}) - \sum_{ i = N_1+1}^N\partial_i (| \partial_i u |^{ p_i-2} \partial_i u )= u^{ p^\star-1}
                    $$ with ${\sigma^1} \cdot \nabla_1 u = | \nabla_1 u|$, 
                     and 
                     \begin{equation}\label{extepsilon}-{ \rm div}_1 ( | \nabla_1 u_\epsilon|^{ \epsilon-1} \nabla_1 u_\epsilon ) -\sum_{ i=N_1+1}^N \partial_i ( | \partial_i u_\epsilon |^{ p_i^\epsilon-2} \partial_i u_\epsilon ) =  u_\epsilon ^{ p_\epsilon^\star-1}
   \end{equation}
  Furthermore $u_\epsilon$ converges tightly to $u$ in $BV^{ \vec p} ( \R^N)$. 
    \end{lemme}
     We do not give the proof of this lemma, which is left to the reader. 

         In the sequel we will consider $u$ and $u_\epsilon$ as in Lemma \ref{lemuueps}.

          \begin{lemme}\label{g(u)}
           Suppose that $u\in BV^{\vec p}$  is  as in      Lemma \ref{lemuueps}.        Suppose that $g $ is Lipschitz continuous on  $\R$, such that $g(0) = 0$ and   $g^\prime \geq 0$, then  $g(u) \in BV^{\vec p}$,  with 
            $\sigma^1 \cdot \nabla_1 (g(u)) = | \nabla_1 (g(u))|$.  Furthermore one has the identity 
             \begin{equation}\label{intgu}\int_{\R^N} | \nabla_1(g(u))| + \sum_{i = N_1+1}^N \int_{\R^N} g^\prime (u) | \partial_i u |^{ p_i} = \int _{\R^N}g(u) u^{ p^\star-1}
             \end{equation} 
                       \end{lemme}
                         \begin{proof}
                         In the following lines, we will use "UTS" to say that the  convergence   holds up to subsequence . 
                          
                      Note that $g( u_\epsilon)\in { \cal D}^{1, \vec {p_\epsilon}}( \R^N)$   by the mean value's theorem,  
                        since $g^\prime\in L^\infty$,  and $(g(u_\epsilon))_{\epsilon} $  is   bounded in that space by the assumptions on $u_\epsilon$, and then also in $BV^{\vec p}_{loc}$. Then since $u_\epsilon$ converges to $u$ almost everywhere "UTS" and $g$ is continuous, $g(u) \in  BV^{ \vec p}( \R^N)$, and $g(u_\epsilon)$ converges weakly to  $g(u)$ in $BV^{\vec p}_{loc}$ "UTS" . In particular it converges to $g(u)$ in $L^q_{loc}$, "UTS" for all $q < p^\star$. 
                         Let us observe that the  sequence of measures $\sigma_\epsilon \cdot \nabla ( g( u_\epsilon))$ converges "UTS"  to $\sigma \cdot \nabla ( g(u))$ : Since $\sigma_\epsilon \cdot \nabla g(u_\epsilon)$ is bounded in $L^1$,   it is sufficient to prove that it converges in the distribution sense. To check this, let $\varphi \in { \cal D} ( \R^N)$,   take $q< p^\star$ so that for $\epsilon$ small enough  $p_i^\epsilon < q$,   then  $\sigma_\epsilon \rightarrow \sigma$ "UTS" in $L^{ q^\prime}_{loc}$.  Using  $g(u_\epsilon)\rightarrow g(u)$ in $L^q_{loc}$ strongly  and "UTS" for all $q < p^\star$,  one has 
                         $\int g(u_\epsilon) \sigma_\epsilon\cdot \nabla \varphi\rightarrow \int  g(u)\sigma \cdot \nabla \varphi$.
                         Secondly  note that                            $u_\epsilon^{ p_\epsilon^\star-1} g(u_\epsilon) \leq |g^\prime|_\infty |u_\epsilon |^{ p_\epsilon^\star}$. By  the strong convergence of $(u_\epsilon)^{ p_\epsilon^\star}$ in $L^1$ one can  suppose that "UTS" is dominated by a function $h$ in $L^1$,  hence so does $u_\epsilon^{ p_\epsilon^\star-1} g(u_\epsilon) $.  By the almost everywhere  convergence "UTS"of $u_\epsilon^{ p_\epsilon^\star-1} g(u_\epsilon)$  to $u^{ p^\star-1} g(u)$ and  the Lebesgue's dominated convergence theorem, one gets that 
                           for any $\varphi \in { \cal D} ( \R^N)$, $\int u_\epsilon^{ p_\epsilon^\star-1} g(u_\epsilon)\varphi \rightarrow \int u^{ p^\star-1} g(u)\varphi$. 
                           We have obtained that 
                    $ \int \sigma_\epsilon \cdot \nabla ( g( u_\epsilon))\varphi  \rightarrow \int\sigma \cdot \nabla ( g(u))\varphi$, for any $\varphi$ in ${ \cal D} ( \R^N)$, hence also  for $\varphi $ in ${ \cal C}_c ( \R^N)$. Furthermore, 
                    by lower semicontinuity one has  for all $\varphi \geq 0$ in ${ \cal C}_c ( \R^N)$, 
                     \begin{eqnarray*}
                     \int | \nabla_1 ( g(u))| \varphi &\leq& \underset{\epsilon \to 0}{\liminf} \int |\nabla_1 ( g(u_\epsilon))|^{1+ \epsilon} \varphi \\
                      &=& \underset{\epsilon \to 0}{\liminf} \int (g^\prime ( u_\epsilon ))^{1+ \epsilon} | \nabla_1 u_\epsilon |^{1+ \epsilon} \varphi 
                   \\
                                          &\leq& \underset{\epsilon \to 0}{\liminf} |g^\prime |_\infty ^\epsilon \int (g^\prime ( u_\epsilon )) | \nabla_1 u_\epsilon |^{1+ \epsilon} \varphi 
                                             \end{eqnarray*}
                      \begin{eqnarray*}\\
                                                                       &=& \underset{\epsilon \to 0}{\liminf}\int \sigma^{1, \epsilon} \cdot \nabla_1 (g(u_\epsilon)) \varphi \\
                                         &=&\int \sigma^1\cdot \nabla_1 (g(u)) \varphi 
                     \end{eqnarray*}
                      This implies since one also  has 
                      $\sigma^1\cdot \nabla_1 (g(u))\leq | \nabla_1 g(u)|$,  that 
                        $\sigma^1 \cdot \nabla_1 (g(u)) = | \nabla_1 (g(u))|$.

                                            To get    identity (\ref{intgu} ) it is then sufficient to multiply  the equation ( \ref{extepsilon}) by $g(u_\epsilon)\varphi$,  and pass to the limit using the previous convergence. Next one can let $\varphi$ go to $1_{\R^N}$ since all the measures involved are bounded measures.

                       \end{proof}
             
       \begin{cor}\label{coruua}
                Let $u$ be as in Lemma \ref{lemuueps}.           
           For any $L$ and $ a>0$, 
            $( u \min ( u^a, L))\in BV^{ \vec p} ( \R^N)$,  
                       $ {\sigma^1} \cdot \nabla_1(  u \min ( u^a, L)) =  | \nabla_1( u \min ( u^a, L)) |$, and 
                       $$  \int | \nabla_1 ( u \min ( u^a, L))| +\sum_{ i = N_1+1}^N\left({1\over 1+ {a\over p_i}}\right)^{p_i-1}  \int |\partial_i (u \min ( u^{ a \over p_i}, L ))|^{p_i} \leq \int u^{ p^\star} \min ( u^a, L).$$
            \end{cor}
            
             \begin{proof}

             We use Lemma \ref{g(u)} with 
             $g(u) = u \min ( u^a, L)$ and equation  (\ref{intgu}).  Then it is sufficient to observe that 
             $$\int g^\prime (u) | \partial_i u |^{p_i} \geq\left(  {1 \over 1+ {a\over p_i}}\right)^{p_i-1}\int | \partial_i ( u \min ( u^{ a\over p_i}, L))|^{p_i}. $$

 \end{proof}

             We now prove the following 
              \begin{prop} \label{propLq}
                Let $u$ be  as in Lemma \ref{lemuueps}, then  $u \in L^\infty$.                \end{prop} 
                \begin{proof}
                            
              This proof follows the lines in \cite{FGK} and \cite {HR2}.  Once more, we reproduce it here for the sake of completeness. 
              We begin to prove that $u \in L^q$ for all $q < \infty$.     In the sequel, $c$ denotes some positive constant which  does not depend
                   on $k$ nor on $a$, which can vary from one line to another. 
                             Let  $k$ to choose later,  and write  for all  $p_j$, ( recall that $p_j = 1$ for $j \leq N_1$) : 
                            
                  \begin{eqnarray*}
                  \int u ^{ p^\star}\min ( u ^{ap_j} , L^{p_j} )&=& \int_{ u\leq k} u^{ p^\star} (\min ( u ^a , L))^{p_j} + \int_{ u \geq k}u^{ p^\star} (\min ( u ^a , L))^{p_j} \\
                  &\leq & k^{ ap_j} \int |u|^{ p^\star} + ( \int _{ u \geq k} u^{p^\star})^{1-{ p_j \over p^\star}} \left(\int ( u \min ( u^a, L))^{p^\star}\right)^{ p_j \over p^\star}.
                  \end{eqnarray*}
                                     Using  the embedding from $BV^{ \vec p}$ in $L^{ p^\star}$ one has 
                                  \begin{equation}\label{ua}
                                  \left( \int ( u \min ( u^a, L))^{p^\star}\right)^{ 1\over p^\star}   \leq c\left( \int | \nabla_1 ( u \min ( u^a, L))|+  \sum_{j=N_1+1} ^N (\int | \partial_j ( u \min ( u^a, L))|^{p_j} )^{1 \over p_j}\right). 
                                   \end{equation} 
                                                                 Using Corollary \ref{coruua},  for $ u \min ( u^{ap_j}, L) $ one gets for all $j$
                                     $$  (1+ a)^{-p_j+1} \int | \partial_j ( u \min ( u^a, L))|^{p_j}  \leq \int u^{ p^\star} \min ( u^{ ap_j}, L^{ p_j})$$
                                      and then   defining                                     $ I_j =  (\int | \partial_j ( u \min ( u^a, L))|^{p_j} )^{1 \over p_j}$ and $\epsilon_k =  \int _{ u \geq k} u^{p^\star}$, 
                                   $$ I_j \leq c(1+a) \left(k^a (\int u^{ p^\star} )^{1 \over p_j} +  \epsilon_k^{{1\over p_j}-{1 \over p^\star}}\left[ \int | \ \nabla_1( u \min ( u^a, L)) |  + \sum_{i=N_1+1}^N I_i\right]\right)$$
                                    and 
                                    $$\int | \ \nabla_1( u \min ( u^a, L)) |  \leq  c(1+a) \left(k^a \int u^{ p^\star} + \epsilon_k^{1-{1 \over p^\star}}\left[ \int | \ \nabla_1( u \min ( u^a, L)) |  + \sum_{i=N_1+1}^N I_i\right]\right). $$
Summing over $j$ one gets 
                                   \begin{eqnarray*}
                                    \int | \ \nabla_1( u \min ( u^a, L)) |&+&\sum _{ j =N_1+1}^N I_j \\
                                    &\leq& c(1+a)\left( k^a \sum_{j=1}^N |u|_{ p^\star } ^{ p^\star  \over p_j} \right. \\
                                    &+& \left. \sum_{j=1}^N \epsilon_k^{{1\over p_j}-{1 \over p^\star}}( \int | \ \nabla_1( u \min ( u^a, L)) |
                                    +\sum_{i=N_1+1}^N I_i)\right). 
                                    \end{eqnarray*}
                                                                       Choosing $k_a$  so that $c(a+1) \sum \epsilon_k^{{1\over p_j}-{1 \over p^\star}}< {1 \over 2}$, ( recall that $p_j < p^\star$ for all $j$ and $\epsilon_k\rightarrow 0$ when $k\rightarrow +\infty$),  we have obtained 
                                   $$ {1\over 2}\left[\int | \ \nabla_1( u \min ( u^a, L)) |+ \sum_{j = N_1+1}^N  I_j \right]\leq  c(1+a) k_a^a \sum_{j=1}^N |u|_{ p^\star } ^{ p^\star  \over p_j} , $$ hence, coming back to (\ref{ua})
                                    $$|u \min ( u^a, L)|_{ p^\star} \leq   c  (1+a) k_a^a \sum _{j=1}^N |u|_{ p^\star } ^{ p^\star  \over p_j} .$$
                                     Letting $L$ go to $\infty$ one gets  
                                     $ |u^{a+1}|_{ p^\star} \leq C^\prime( |u|_{p^\star} ) (1+a) k_a ^a ,  $
                                      taking the power ${1 \over a+1}$, one has obtained that for $q = p^\star ( a+1)$,  
                                        $$|u|_q \leq C^\prime( |u|_{ p^\star})^{1\over 1+ a} (1+a)^{1 \over a+1} k_a ^{a\over a+1}, $$
                                         and then $u$ belongs to $L^q$ for all $q < \infty$. 
                                       
                                       \bigskip
                                         
                                          To  prove that $u\in L^\infty$, we still follow the lines in \cite{HR2}. 
                                         
                                          Choose $q > p^\star$ so that 
                                          $ \epsilon := {-1 \over p^\star} + ( 1-{ p^\star\over q}) ( 1-{1 \over p^\star}){1 \over p^+-1} >0$. 
                                                                              Let 
                                           $\varphi_k =  ( u-k)_+$,   and 
                                            $A_k = \{ x, u(x) > k\}$. 
                                            Let us begin to note that $A_k$ is of finite measure   for all  $k>0$, since 
                                             $$| \{ x, u(x) > k\}| k^{p^\star} \leq \int_{ u> k} |u|^{p^\star}  \leq |u|^{p^\star}_{p^\star}.$$
                                            We then  deduce that for $k>0$, 
                                              $ (u-k)_+\in L^1$, since 
                                              \begin{equation}\label{uk}
                                              \int (u-k)^+ \leq \int_{ u \geq k} u\leq \int_{ u \geq k} { u^{ p^\star} \over k^{ p^\star-1}}.
                                              \end{equation}
                                            We now apply Lemma \ref{g(u)} with $g(u) = (u-k)^+$.                                                
                                    Using (\ref{intgu})  one gets 
                                    \begin{eqnarray*}
                                     | \nabla_1 \varphi_k|_1+ \sum_{ i = N_1+1}^N  | \partial_i \varphi_k|_{p_i}^{ p_i} &=& \int  u^{ p^\star-1} ( u-k)^+ 
                                   \\
                                   & \leq& |u|_q^{ p^\star-1}  |A_k |^{(1-{ p^\star \over q})(1-{1 \over p^\star})} |\varphi_k |_{ p^\star}\\
                                   &\leq& c |A_k |^{(1-{ p^\star \over q})(1-{1 \over p^\star})} |\varphi_k |_{ p^\star}. 
                                   \end{eqnarray*}
                                    We then have since $| \varphi_k|_{ p^\star} \leq |u|_{ p^\star}= 1$, by Lemma  \ref{lem1}
                                    \begin{eqnarray*}
                                     |\varphi_k |_{ p^\star} ^{ p^+} &\leq&c \left(   | \nabla_1 \varphi_k|_1 + \sum_{ i =N_1+1}^N{1\over p_i}  | \partial_i \varphi_k |_{ p_i}^{ p_i} \right)\\
                                     &\leq& c|A_k |^{(1-{ p^\star \over q})(1-{1 \over p^\star})} |\varphi_k |_{ p^\star}.
                                     \end{eqnarray*}
                                  hence 
                                  
                                            $$ | \varphi_k |_{p^\star} \leq c |A_k|^{\epsilon + { 1 \over p^\star}}, $$and using  H\" older's inequality,   one derives 
                                            $\int_{ \R^N} (u-k)_+ \leq |A_k |^{1-{1 \over p^\star} } | \varphi_k |_{p^\star} \leq c |A_k |^{1+ \epsilon}$. 
                                            
                                            Let $y(k) = \int_k^\infty |A_\tau| d\tau$, then 
                                              $y(k) =  \int_{ \R^N}( u-k)_+\leq c ( -y^\prime (k))^{1+ \epsilon}$, 
                                                                                             and integrating one obtains 
                                               $$-y^{ \epsilon \over 1+ \epsilon} ( u(s)) + y^{ \epsilon \over 1+ \epsilon} (k) \geq { \epsilon \over 1+ \epsilon} c^{-1\over 1+ \epsilon} ( u(s)-k)$$
                                               hence for any $s$, recalling ( \ref{uk}), for some constants $b$ and $\gamma>0$:
                                                $$ u(s)-k \leq {1+ \epsilon \over \epsilon} c^{1 \over 1+ \epsilon}{ |u|_{p^\star}^{ p^\star \epsilon \over 1+ \epsilon} \over k^{ (p^\star-1) \epsilon \over 1+ \epsilon}}\leq {b\over k^\gamma}.$$
                                                 Optimizing with respect to $k$, ie taking the infimum one gets that 
                                                 $$ u(s) \leq c ( |u|_{p^\star})$$
                                                   \end{proof}
                                                   \begin{rema}

Let $q_1,\cdots, q_m$ be such that $\{p_{N_1+1},\cdots,p_N\}=\{q_1,\cdots, q_m\}$, and $q_i\neq q_j$ when $i\neq j$.

                                                   Note that one could consider in place of ${\cal K}  = \inf_{ u \in { \cal D}^{1, \vec p}, |u|_{ p^\star}=1}  | \nabla _1 u |_1 + \sum_{i = N_1+1}^N { 1 \over p_i } |\partial_i u |^{ p_i}_{p_i}$ the infinimum 
                                                   $$\tilde {\cal K}  = \inf_{ u \in { \cal D}^{1, \vec p}(\R^N), |u|_{ p^\star}=1}  | \nabla _1 u |_1 + \sum_{j=1}^m({1\over q_j} \int (\sum_{i, p_i = q_j}|\partial_i u |^2)^{  q_j\over 2})$$ and prove the existence of an extremal function with obvious changes.
                                                    \end{rema}

$$$$

Fran\c coise Demengel : francoise.demengel@u-cergy.fr

Thomas Dumas : thomas.dumas@u-cergy.fr


\begin{thebibliography}{99}
  \bibitem{AFTL}A. Alvino, V. Ferone, G. Trombetti, and P-L. Lions, { \em Convex symmetrization and applications} , Ann. Inst. H. Poincar\'e,  Anal. Non Lin\'eaire 14 (1997), no. 2, 275-293.
          \bibitem{Anz}  { G. Anzellotti}, { \em Pairings between measures and bounded functions and compensated compactness,} Ann. Mat. Pura Appl. (4) 135 (1) (1983) 293-318.

       \bibitem{A} T. Aubin, {\em Probl\`emes isop\'erim\'etriques et espaces de Sobolev} , J. Differential Geometry 11 (4) (1976),  573-598.

  \bibitem  {BMS} L.  Boccardo, P. Marcellini, C.  Sbordone, {\em  $L^\infty$ -regularity for variational problems with sharp nonstandard growth conditions},  Boll. Un. Mat. Ital. A (7) 4 (1990), no. 2, 219-225.  
  


\bibitem{BL} H. Brezis, E. Lieb, { \em A relation between pointwise convergence of functions and convergence of functionals,}  Proc. Amer. Math. Soc. 88 (1983)
486-490.

\bibitem{CNV}  D. Cordero-Erausquin,  B. Nazaret, and C. Villani, {\em A  mass-transportation approach to sharp Sobolev and Gagliardo-Nirenberg inequalities}
                                                 Advances in Mathematics 182,  (2004) 307-332. 


\bibitem{CMM}G. Cupini, P.  Marcellini,  E.  Mascolo,  {\em Regularity under sharp anisotropic general growth conditions},  Discrete and Continuous dynamical systems, Volume 11, Number 1, ( 2009),  pp. 67-86.



       
        \bibitem{FD1}  F. Demengel, {\em 
On Some Nonlinear Partial
Differential Equations involving the
"1"-Laplacian and Critical Sobolev Exponent}, 
ESAIM: Control, Optimisation and Calculus of Variations,
December (1999), Vol 4, p 667-686.

\bibitem{FD2} 
F. Demengel
{\em Some Existence's results  for
 non coercive "$1$-Laplacian" operator},  Asymptotic Analysis 43 (2005), 287-322. 

\bibitem{FD3} 
F. Demengel
{\em Functions locally $1$-harmonic}, 
Applicable analysis,  83, (2004), n. 9, 343-366. 
 \bibitem{DT} F. Demengel, R. Temam, { \em Functions of a measure and its applications} ,  Indiana Math. Journal,33 (5), 1984, 673-700.




\bibitem{Di} J. Dieudonn\'e, { \em El\'ements d'analyse} Tome 2, Gauthiers Villars. 


 \bibitem{ThD} T. Dumas, { \em Existence de solutions pour des \'equations apparent\'ees au 1 Laplacien anisotrope}  Th\`ese d'Universit\'e, Universit\'e de Cergy, 2018.  

      \bibitem{ELM1} L. Esposito, F. Leonetti, G. Mingione,  {\em Higher Integrability for Minimizers of Integral Functionals With $(p, q)$ Growth}, J.  Differential equations , 157, (1999), 414-438.
 \bibitem{ELM2} L. Esposito, F. Leonetti, G. Mingione,  {\em Sharp regularity for functionals with $(p,q)$ growth} 
J. Differential Equations 204 (2004) 5-55. 
\bibitem{FGK} I. Fragala, F. Gazzola, B. Kawohl, { \em Existence and nonexistence results for anisotropic quasilinear elliptic equation} , Ann. I. H. Poincar\'e Anal., 21,  (2004) 715-734.


      
 

 
  \bibitem{FS}  N. Fusco, C. Sbordone, {\em Local boundedness of minimizers in a limit case} , Manuscripta Math., 69 (1990), 19-25.   



\bibitem{G} M. Giaquinta, {\em Growth conditions and regularity, a counterexample} , Manuscripta Math., 59 (1987), 245-248.

\bibitem{Gi} E. Giusti,   Minimal surfaces and functions of bounded variation, Birkhauser, 1984. 

\bibitem {HR1} A. El Hamidi, J.M. Rakotoson, { \em Compactness and quasilinear problems with critical exponents} , Differential Integral Equations 18 (2005)
1201-1220.

   \bibitem{HR2}{ A. El Hamidi, J.M. Rakotoson, }{\em Extremal functions for the anisotropic Sobolev inequalities, }Ann. I.H. Poincar\'e, Analyse non lin\'eaire,  24 (2007) 741-756.
        

\bibitem{KrKo}S.N. Kruzhkov, I.M. Kolodii, {\em On the theory of embedding of anisotropic Sobolev spaces}, Russian Math. Surveys 38 (1983) 188-189.

\bibitem{KruKoro} S.N. Kruzhkov, A.G. Korolev, {\em On the imbedding theory of anisotropic function spaces,} Soviet Math. Dokl. Vol.32 (1985), No.3, 829-832.



 \bibitem{PL1}  P.L. Lions, { \em The concentration-compactness principle in the calculus of variations . The limit case, part 1,}  Rev. Mat. Iberoamericana 1 (1) (1985) 145-201;

 \bibitem{PL2}  P.L. Lions,  { \em The concentration-compactness principle in the calculus of variations. The limit case, part 2} , Rev. Mat. Iberoamericana 1 (2) (1985)
45-121.  


\bibitem{M1} P. Marcellini, {\em Regularity of minimizers of integrals in the calculus of variations with non standard 
growth conditions} , Arch. Rational Mech. Anal., 105 (1989), 267-284.

\bibitem{M2} P. Marcellini, {\em Regularity and existence of solutions of elliptic equations with $p- q$ growth conditions}, J. Differential Equations, 90 (1991), 1-30

    
\bibitem{MRST}{A. Mercaldo, J.D. Rossi, S. Segura de Le\'on, C. Trombetti,} {\em Anisotropic p,q-Laplacian equations when p goes to 1}, Nonlinear Analysis, 73 (2010) 3546- 3560.

\bibitem{Ni} S.M. Nikolskii, {\em On imbedding, continuation and approximation theorems for differentiable functions of several variables}, 
(Russian) Uspehi Mat. Nauk 16 1961 no. 5 (101), 63-114. 



\bibitem{ST} G. Strang, R. Temam, {\em Duality and relaxation in the variational problems  of plasticity} J. M\'ecanique, 1980, 19(3), 493-527. 

  \bibitem{Ta} G. Talenti 
{ \em Best constant in Sobolev inequality} , Ann. Mat. Pura Appl. (4) 110 (1976), 353-372.

\bibitem{Tem} R. Temam, Mathematical problems in plasticity, Gauthiers Villars, 1993. 


 \bibitem {T} M. Troisi { \em Teoremi di inclusione per spazi di Sobolev non isotropi} , Ricerche Mat. 18 (1969) 3-24.
      \bibitem{V} { J. Vetois,} 
{ \em Decay estimates and a vanishing phenomenon for the solutions of critical anisotropic equations} . Adv. Math. 284 (2015), 122- 158.
\end{thebibliography}
                \end{document}